\documentclass[11pt,a4paper]{article}
\usepackage{enumerate,amsmath,amsthm,amsfonts,float,graphicx,color,pst-all,lineno}
\usepackage[numbers,sort&compress]{natbib}
\usepackage[lmargin=28mm,rmargin=28mm,tmargin=36mm,bmargin=26mm]{geometry}
\renewcommand{\baselinestretch}{1.1}
\usepackage[colorlinks=true,citecolor=black,linkcolor=black,urlcolor=blue]{hyperref}
\newcommand{\doi}[1]{\href{http://dx.doi.org/#1}{\texttt{doi:#1}}}
\newcommand{\arxiv}[1]{\href{http://arxiv.org/abs/#1}{\texttt{arXiv:#1}}}

\newcommand{\ceil}[1]{\ensuremath{\protect\lceil#1\rceil}}

\theoremstyle{plain}
\newtheorem{theorem}{Theorem}
\newtheorem{lemma}[theorem]{Lemma}

\theoremstyle{definition}

\begin{document}

\title{\bf\vspace*{-4ex} Colouring the Triangles Determined by a Point
  Set}

\author{Ruy Fabila-Monroy\footnote{Departamento de Matem\'aticas,
    Cinvestav, Distrito Federal, M\'exico
    (\texttt{ruyfabila@math.cinvestav.edu.mx}). Supported by an
    Endeavour Fellowship from the Australian Government.}  \and
  David~R.~Wood\footnote{Department of Mathematics and Statistics, The
    University of Melbourne, Melbourne, Australia
    (\texttt{woodd@unimelb.edu.au}). Supported by a QEII Fellowship
    from the Australian Research Council.}}

\date{\today}

\maketitle

\begin{abstract}
  Let $P$ be a set of $n$ points in general position in the plane. We
  study the chromatic number of the intersection graph of the open
  triangles determined by $P$. It is known that this chromatic number
  is at least $\frac{n^3}{27}+O(n^2)$, and if $P$ is in convex
  position, the answer is $\frac{n^3}{24}+O(n^2)$. We prove that for
  arbitrary $P$, the chromatic number is at most
  $\frac{n^3}{19.259}+O(n^2)$.
\end{abstract}

\section{Introduction}
\label{sec:Intro}

Let $P$ be a set of $n$ points in general position in the plane (that
is, no three points are collinear). A triangle with vertices in $P$ is
said to be \emph{determined by} $P$.  Let $G_P$ be the intersection
graph of the set of all open triangles determined by $P$. That is, the
vertices of $G_P$ are the triangles determined by $P$, where two
triangles are adjacent if and only if they have an interior point in
common. This paper studies the chromatic number of $G_P$.

Consider a colour class $X$ in a colouring of $G_P$. Then $X$ is a set
of triangles determined by $P$, no two of which have an interior point
in common. If $P'\subseteq P$ is the union of the vertex sets of the
triangles in $X$, then there is a triangulation of $P'$ in which each
triangle in $X$ is a face. The converse also holds: the set of faces
in a triangulation of a subset of $P$ can all be assigned the same
colour in a colouring of $G_P$. Thus $\chi(G_P)$ can be considered to
be the minimum number of triangulations of subsets of $P$ that cover
all the triangles determined by $P$, where a triangulation $T$ covers
each if its faces.

% Let $X_P$ be the set of all open triangles determined by $P$.  Let
% $G_P$ be the intersection graph of $X_P$. That is, $V(G_P)=T_P$,
% where two triangles are adjacent in $G_P$ if their interiors have a
% point in common. We study the chromatic number of $G_P$. Each colour
% class in a colouring of $G_P$ is a set of triangles, no two of which
% have an interior point in common. Such a set of triangles can
% extended to a triangulation of $P$. Thus a colouring of $G_P$ can be
% considered to be a covering of $X_P$ by triangulations (where a
% triangle in a triangulation is not necessarily a face).

First consider $\chi(G_P)$ for small values of $n$.  If $n=3$ then
$\chi(G_P)=1$ trivially. If $n=4$ then $\chi(G_P)=2$, as illustrated
in Figure~\ref{fig:FourPoints}. If $n=5$ then $\chi(G_P)=5$, as
illustrated in Figure~\ref{fig:FivePoints}.

\begin{figure}[H]
  \begin{center}\includegraphics{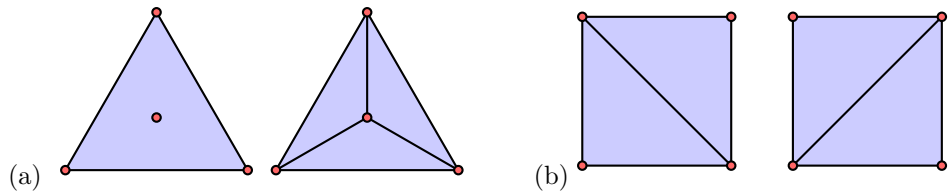}\end{center}
  \vspace*{-4ex}
  \caption{\label{fig:FourPoints} Colouring the triangles determined
    by four points: (a) non-convex position, (b) convex position. In
    both cases, $\chi(G_P)=\omega(G_P)=2$.}
\end{figure}

\begin{figure}
  \begin{center}\includegraphics{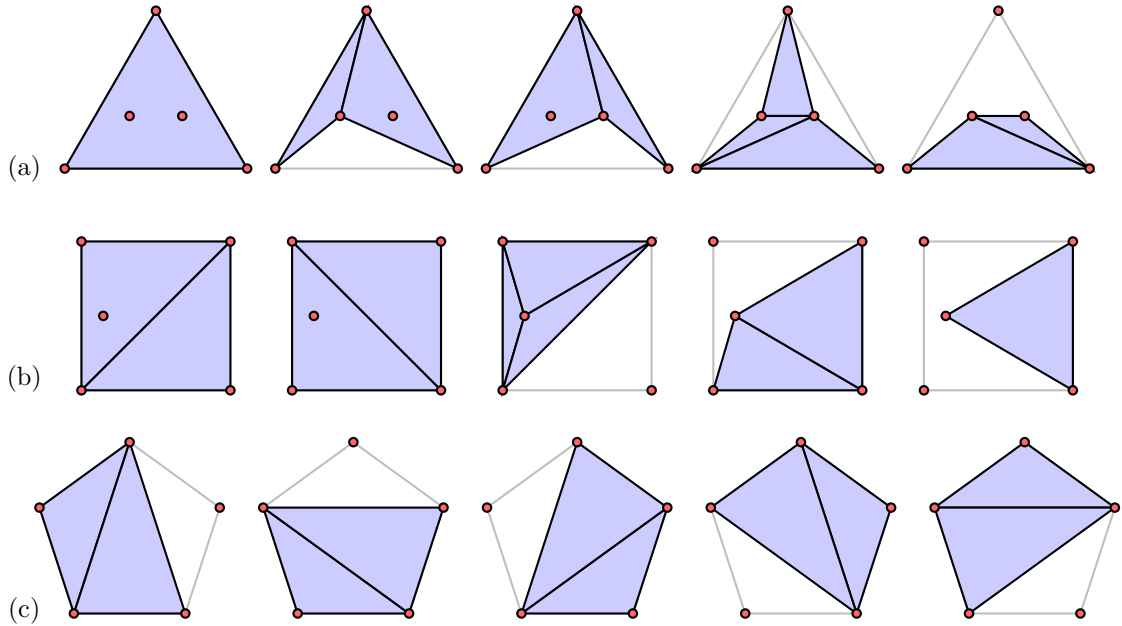}\end{center}
  \vspace*{-4ex}
  \caption{\label{fig:FivePoints} Colouring the triangles determined
    by five points: (a) three boundary points, (b) four boundary
    points, (c) five boundary points. In each case,
    $\chi(G_P)=\omega(G_P)=5$.}
\end{figure}

For $n=6$, we used the database of 16 distinct order types of 6 points
in general position \citep{AAK02}, and calculated $\chi(G_P)$ exactly
for each such set using sage \citep{sage}. As shown in
Appendix~\ref{SixPoints}, $\chi(G_P)=8$ for each 6-point set $P$. This
result will also be used in the proof of
Theorem~\ref{thm:ColourOpenTriangles} below.

It is interesting that $\chi(G_P)$ is invariant for sets of $n$
points, for each $n\leq 6$. However, this property does not hold for
$n=7$. If $P$ consists of 7 points in convex position, then
$\chi(G_P)=14$, whereas we have found a set $P$ of 7 points in general
position for which $\chi(G_P)=13$; see Appendix~\ref{sec:SevenPoints}.

Now consider $\chi(G_P)$ for arbitrarily large values of $n$.  If $P$
is in convex position then the problem is solved: results of
\citet{Cano} imply that
$$\chi(G_P)=
      \begin{cases}
        \tfrac{1}{24}\,(n-1)n(n+1)&\text{ if $n$ is odd}\\
        \tfrac{1}{24}\,(n-2)n(n+2)&\text{ if $n$ is even}\enspace.
      \end{cases}
$$
See Appendix~\ref{sec:Related} for a proof of this and other related
results.

Our main contribution is to prove the following bound for arbitrary
point sets, where $\omega(G_P)$ is the maximum order of a clique in
$G_P$.

\begin{theorem}
  \label{thm:ColourOpenTriangles}
  For every set $P$ of $n$ points in general position in the plane,
 $$\frac{n^3}{27}\;\leq\;\omega(G_P)\;\leq\;\chi(G_P)\;\leq\; 
 \frac{27n^3}{520}+O(n^2)=\frac{n^3}{19.259\ldots}+O(n^2)\enspace.$$
\end{theorem}

\section{Proof of Theorem~\ref{thm:ColourOpenTriangles}}

The lower bound in Theorem~\ref{thm:ColourOpenTriangles} follows
immediately from a theorem by \citet{BF-GD84}, who proved that for
every set
$P$ of $n$ points in general position, there is a point
$q$ in the plane such that $q$ is in the interior of at least
$\frac{n^3}{27}+O(n^2)$ triangles determined by $P$. These triangles
form a clique in $G_P$, implying $\chi(G_P)\geq
\omega(G_p)\geq\frac{n^3}{27}+O(n^2)$.  This result is called the
`first selection lemma' by \citet[Section~9.1]{Mat02}. See
\citep{Bukh-EJC06} for an alternative proof and see
\citep{Barany82,FGLNP} for generalisations. 
% BFL90,BL92,SS-CPC04,ACEGSW-DCG91

Note that Boros and F{\"u}redi's theorem is stronger than simply
saying that $\omega(G_p)\geq\frac{n^3}{27}+O(n^2)$. For example, for
sets of $n$ points in convex position, $G_P$ is invariant. Moving the
points around a circle does not change the graph, which is not true
for the question of a point in many triangles. Indeed, \citet{BMN}
proved that there is a set $P$ of $n$ points in convex position, such
that every point in the plane is in the interior of at most
$\frac{n^3}{27}+O(n^2)$ triangles determined by $P$ (thus proving that
the Boros-F\"uredi bound is best possible). However, in this case,
$\omega(G_P)=\frac{n^3}{24}+O(n^2)$ by the result of \citet{Cano}
mentioned above.

It is an interesting open problem whether the lower bound on
$\chi(G_p)$ in Theorem~\ref{thm:ColourOpenTriangles} is tight. That
is, are there infinitely many $n$-point-sets $P$ for which $\chi(G_P)=\frac{n^3}{27}+O(n^2)$?

The proof of the upper bound in Theorem~\ref{thm:ColourOpenTriangles}
depends on the following lemma.

\begin{lemma}
  \label{lem:ColourSeparatedOpenTriangles}
  Let $A$ and $B$ be sets of $n$ points in general position in the
  plane separated by a line. Let $X$ be the set of open triangles that
  are determined by $A\cup B$ and have at least one vertex in each of
  $A$ and $B$.  Then the chromatic number of the intersection graph of
  $X$ is at most $\frac{2}{5}n^3+O(n^2)$
\end{lemma}

\begin{proof}
  We proceed by induction on $n$. It is easily seen that two colours
  suffice for $n\leq 2$.

  If necessary, add a point to $A$ and $B$ so that $|A|=|B|=2m$,
  where $m:=\ceil{\frac{n}{2}}$.  Adding points cannot decrease the
  chromatic number.  By the Ham Sandwich Theorem there is a line
  $\ell$ such that in each open half-plane determined by $\ell$, there
  are exactly $m$ points of $A$ and $m$ points of $B$. Without loss of
  generality, $\ell$ is horizontal.  Let $A_1$ and $A_2$ respectively
  be the subsets of $A$ consisting of points above and below
  $\ell$. Define $B_1$ and $B_2$ analogously.  Thus
  $|A_1|=|A_2|=|B_1|=|B_2|=m$. We call $A_1$, $A_2$, $B_1$ and $B_2$
  \emph{quadrants}.

  Let $G$ be the complete 4-partite graph with colour classes
  $A_1,A_2,B_1,B_2$. \citet{FabilaWood-Decomp} proved that there is a
  set of $m^3+O(m^2)$ copies of $K_4$ in $G$ such that each triangle
  of $G$ appears in some copy. Say $\{a_1,a_2,b_1,b_2\}$ induce such a
  copy of $K_4$, where $a_i\in A_i$ and $b_i\in B_i$. The intersection
  graph of the open triangles determined by any set of four points is
  2-colourable, as illustrated in Figure~\ref{fig:FourPoints}. Thus
  $2m^3+O(m^2)$ colours suffice for the triangles with vertices in
  distinct quadrants.

  For each $i,j\in\{1,2\}$, by induction, $\frac{2}{5}m^3+O(m^2)$
  colours suffice for the triangles in $X$ determined by $A_i\cup
  B_j$. Moreover, the triangles determined by $A_1\cup B_1$ can share
  the same set of colours as the triangles determined by $A_2\cup
  B_2$.  Thus $\frac{6}{5}m^3+O(m^2)$ colours suffice for the
  triangles with vertices in two quadrants.  This accounts for all
  triangles in $X$. The total number of colours is
  $(2+\frac{6}{5})m^3+O(m^2)=\frac{2}{5}n^3+O(n^2)$.
\end{proof}

% \begin{theorem}
%   \label{thm:ColourOpenTriangles}
%   %   For every set $P$ of $n$ points in general position in the
%   %   plane,
%   %   there is a set $\mathcal{X}$ of
%   %   $\frac{27n^3}{520}+O(n^2)=\frac{n^3}{19.259\ldots}+O(n^2)$
%   %   non-crossing geometric graphs with vertex set contained in
%   %   $P$,
%   %   such that each triangle determined by $P$ is a face of some
%   %   graph
%   %   in
%   %   $\mathcal{X}$.\\
%   %   OR\\
%   For every set $P$ of $n$ points in general position in the plane,
%   the chromatic number of the intersection graph of the set of open
%   triangles determined by $P$ is at most
%   $\frac{27n^3}{520}+O(n^2)=\frac{n^3}{19.259\ldots}+O(n^2)$.
% \end{theorem}

\begin{proof}[Proof of the Upper Bound in Theorem~\ref{thm:ColourOpenTriangles}] 
  We proceed by induction on $n$.  As shown in
  Section~\ref{sec:Intro}, for $n=3,4,5,6$ every point set $P$
  with $|P|=n$ satisfies $\chi(G_P)=1,2,5,8$ respectively. 
Now assume that $n\geq 7$.

  \citet{Ceder64} proved that there are three concurrent lines that
  divide the plane into six parts each containing at least
  $\frac{n}{6}-1$ points in its interior; also see
  \citep{Bukh-EJC06}. So each part has at most $m:=\frac{n}{6}+5$
  parts. Add points if necessary so that each part contains exactly
  $m$ points. Adding points cannot decrease the chromatic number.  Let
  $P_1,P_2,\dots,P_6$ be the partition of $P$ determined by the six
  parts, in clockwise order about the point of concurrency. Each $P_i$
  is called a \emph{sector}. Let $G$ be the complete 6-partite graph,
  with colour classes $P_1,P_2,\dots,P_6$.

  \citet{FabilaWood-Decomp} proved that there is a set of $m^3+O(m^2)$
  copies of $K_6$ in $G$ such that each triangle appears in some
  copy. Each copy of $K_6$ corresponds to a set of points
  $\{x_1,\dots,x_6\}$ such that each $x_i\in P_i$. The chromatic
  number of the intersection graph of open triangles determined by
  $\{x_1,\dots,x_6\}$ is at most 8; see Appendix~\ref{SixPoints}. Thus
  $8m^3+O(m^2)$ colours suffice for the triangles determined by $P$
  with vertices in distinct sectors.

  For $i,j\in\{1,\dots,6\}$, let $X_{i,j}$ be the set of triangles
  determined by $P_i\cup P_j$ that have at least one endpoint in each
  of $P_i$ and $P_j$.

  By induction, $\frac{27}{520}(2m)^3+O(m^2)$ colours suffice for the
  triangles determined by $P_1\cup P_2$. The same set of colours can
  be used for the triangles determined by $P_3\cup P_4$, and for the
  triangles determined by $P_5\cup P_6$. This accounts for all
  triangles contained in a single sector, as well as $X_{1,2}\cup
  X_{3,4}\cup X_{5,6}$.

  We now colour $X_{i,j}$ for other values of $i,j$.  Note that $P_i$
  and $P_j$ are separated by a line. Thus, by
  Lemma~\ref{lem:ColourSeparatedOpenTriangles},
  $\frac{2}{5}m^3+O(m^2)$ colours suffice for the triangles in
  $X_{i,j}$. Moreover, $X_{2,3}\cup X_{4,5}\cup X_{6,1}$ can use the
  same set of colours, as can $X_{1,5}\cup X_{2,4}$ and $X_{1,3}\cup
  X_{4,6}$ and $X_{3,5}\cup X_{2,6}$.  Each of $X_{1,4}$, $X_{2,5}$
  and $X_{3,6}$ use their own set of colours.  In total the number of
  colours is
   $$8m^3+O(m^2) 
   \,+\, \frac{27}{520}(2m)^3+O(m^2) \,+\, \frac{14}{5}m^3+O(m^2)
   % \,=\, (8+\frac{27}{520}+\frac{14}{5})m^3+O(m^2) \,=\,
   % (\frac{4160}{520}+\frac{216}{520}+\frac{1456}{520})m^3+O(m^2)
   % \,=\, (\frac{5832}{520})m^3+O(m^2) \,=\,
   % (\frac{5832}{520})(\frac{n}{6})^3+O(n^2)
   \,=\, \frac{27}{520}n^3+O(n^2)$$
 \end{proof}

\def\cprime{$'$} \def\soft#1{\leavevmode\setbox0=\hbox{h}\dimen7=\ht0\advance
  \dimen7 by-1ex\relax\if t#1\relax\rlap{\raise.6\dimen7
  \hbox{\kern.3ex\char'47}}#1\relax\else\if T#1\relax
  \rlap{\raise.5\dimen7\hbox{\kern1.3ex\char'47}}#1\relax \else\if
  d#1\relax\rlap{\raise.5\dimen7\hbox{\kern.9ex \char'47}}#1\relax\else\if
  D#1\relax\rlap{\raise.5\dimen7 \hbox{\kern1.4ex\char'47}}#1\relax\else\if
  l#1\relax \rlap{\raise.5\dimen7\hbox{\kern.4ex\char'47}}#1\relax \else\if
  L#1\relax\rlap{\raise.5\dimen7\hbox{\kern.7ex
  \char'47}}#1\relax\else\message{accent \string\soft \space #1 not
  defined!}#1\relax\fi\fi\fi\fi\fi\fi}

 % \newpage
\appendix
\section{Related Results}
\label{sec:Related}

The following theorem is obtained by combining results by
\citet{BF77,BF-GD84} and \citet{Cano}. In particular,
\citet{BF77,BF-GD84} proved that (A) = (B) = (F) and \citet{Cano}
proved that (E) = (F). We include the proof for completeness. See
\citep{A006918} for other combinatorial objects counted by the same
formula. A \emph{tournament} is an orientation of a complete graph.

\begin{theorem}
  \label{thm:Related}
  The following are equal:
  \begin{enumerate}[(A)]
  \item the maximum number of directed 3-cycles in a tournament on $n$
    vertices,
  \item the maximum number of triangles determined by $n$ points in
    general position with an interior point in common,
  \item the maximum number of triangles determined by $n$ points in
    convex position with an interior point in common,
  \item the clique number of the intersection graph of the open
    triangles determined by $n$ points in convex position,
  \item the chromatic number of the intersection graph of the open
    triangles determined by $n$ points in convex position,
  \item
    \begin{equation*}
      \begin{cases}
        \tfrac{1}{24}\,(n-1)n(n+1)&\text{ if $n$ is odd}\\
        \tfrac{1}{24}\,(n-2)n(n+2)&\text{ if $n$ is even}\enspace.
      \end{cases}
    \end{equation*}
  \end{enumerate}
\end{theorem}

\newcommand{\AllBold}[1]{\textbf{\boldmath #1}}

\begin{proof}
  \AllBold{(A) $\leq$ (F):} (This is an exercise in
  \citep[page~33]{Anderson89}.)\ Let $G$ be a tournament on $n$
  vertices. Let $X$ be the set of directed 3-cycles in $G$. For each
  triple $\{u,v,w\}$ of vertices not in $X$, exactly one of $u,v,w$,
  say $u$, has outdegree $2$ in $G[\{u,v,w\}]$. Charge this triple to
  $u$.  Exactly $\binom{\deg^+(u)}{2}$ such triples are charged to
  $u$.  Thus the number of triples not in $X$ equals
  $\sum_u\binom{\deg^+(u)}{2}$.  Hence
  \begin{equation}
    \label{Xequals}
    |X|=\binom{n}{3}-\sum_u\binom{\deg^+(u)}{2}\enspace,
  \end{equation}
  which is maximised when the outdegrees are as equal as possible
  (subject to $\sum_u\deg^+(u)=\binom{n}{2}$). Thus when $n$ is odd,
  $|X|$ is maximised when every vertex has outdegree
  $\frac{n-1}{2}$. Hence $|X|\leq\binom{n}{3}-n\binom{(n-1)/2}{2}=
  \tfrac{1}{24}\,(n-1)n(n+1)$. When $n$ is even, $|X|$ is maximised
  when half the vertices have outdegree $\frac{n-2}{2}$ and the other
  half have outdegree $\frac{n}{2}$. Hence
  $|X|\leq\binom{n}{3}-\frac{n}{2}\binom{(n-2)/2}{2}-\frac{n}{2}\binom{n/2}{2}=
  \tfrac{1}{24}\,(n-2)n(n+2)$.

  \medskip\AllBold{(B) $\leq$ (A):} Let $P$ be a set of $n$ points in
  general position.  Let $X$ be a set of triangles determined by $P$
  that contain a common interior point $q$. Let $G$ be the $n$-vertex
  tournament with vertex set $P$, where the edge $vw$ is directed from
  $v$ to $w$ whenever $w$ is clockwise from $v$ in the triangle
  $vwq$. if $vwq$ are collinear then orient $vw$ arbitrarily in $G$.
  A triangle in $X$ is a directed 3-cycle in $G$.  Thus $|X|$ is at
  most the maximum number of directed 3-cycles in an $n$-vertex
  tournament.

  \medskip \AllBold{(C) $\leq$ (B):} This follows immediately from the
  definitions.

  \medskip\AllBold{(C) $\leq$ (D):} If $P$ is a set of points, and $X$
  is a set of triangles determined by $P$ with an interior point in
  common, then $X$ is a clique in $G_P$. Thus (D) $\geq$ (C).

  \medskip\AllBold{(D) $\leq$ (E):} The chromatic number of every
  graph is at least its clique number.

  \medskip\AllBold{(E) $\leq$ (D):} For sets $P$ of $n$ points in
  convex position, $G_P$ does not depend on the particular choice of
  $P$. Thus we may assume that $P$ consists of $n$ equally spaced
  points around a circle. Below we define a specific point $q$ at or
  near the centre of the circle. Say a triangle determined by $P$ is
  \emph{central} if it contains $q$ in its interior. Thus the set of
  central triangles are a clique in $G_P$. For each central triangle
  $uvw$, we define an independent set of triangles (including $uvw$)
  that is said to \emph{belong} to $uvw$. We prove that each triangle
  is in an independent set belonging to some central triangle. Thus
  these independent sets define a colouring of $G_P$, with one colour
  for each central triangle.

  First suppose that $n$ is even. For each point $v\in P$, let $v'$ be
  the point on the circle antipodal to $v$. Since $n$ is even, $v'\in
  P$. A triangle determined by $P$ is \emph{long} if it contains an
  antipodal pair of vertices.  Let $q$ be a point near the centre of
  the circle, such that for all consecutive points $v,w\in P$, exactly
  one of the long triangles $vv'w$ and $vv'w'$ contain $q$ in their
  interior.  If $uvw$ is a non-long central triangle, then each of
  $uvw'$, $uv'w$ and $u'vw$ is not central, and
  $\{uvw,uvw',uv'w,u'vw\}$ is the independent set that belongs to
  $uvw$.  If $vv'w$ is a long central triangle, then $vv'w'$ is not
  central, and $\{vv'w,vv'w'\}$ is the independent set that belongs to
  $vv'w$.  We claim that every triangle determined by $P$ is in an
  independent set that belongs to a central triangle. Let $uvw$ be a
  non-central triangle. Without loss of generality, $vw$ separates $u$
  from $q$, implying $u'vw$ is a central triangle, and $uvw$ is in the
  independent set that belongs to $u'vw$ (regardless of whether $u'vw$
  is long), as claimed.

  Now assume that $n$ is odd. For each point $v\in P$, let $v'$ be the
  point in $P$ immediately clockwise from the point on the circle
  antipodal to $v$ (which is not in $P$ since $n$ is odd).  Let $q$ be
  the centre of the circle.  If $uvw$ is a central triangle, and no
  two of $u,v,w$ are consecutive around the circle, then each of
  $uvw'$, $uv'w$ and $u'vw$ is not central, and
  $\{uvw,uvw',uv'w,u'vw\}$ is the independent set in $G_P$ that
  belongs to $uvw$.  If $uvw$ is a central triangle, and $u$ and $v$
  are consecutive, then $uv'w$ and $u'vw$ are not central, and
  $\{uvw,uv'w,u'vw\}$ is the independent set in $G_P$ that belongs to
  $uvw$.  We claim that every triangle determined by $P$ is in an
  independent set that belongs to a central triangle. Let $uvw$ be a
  non-central triangle. Without loss of generality, $vw$ separates $u$
  from $q$. Let $x$ be the vertex immediately anticlockwise from
  $u'$. Then $xvw$ is a central triangle, and $x'=u$. Thus $uvw$ is in
  the independent set that belongs to $xvw$, as claimed.

  Since there is one colour for each central triangle in the above
  colouring, the set of central triangles are a maximum clique in
  $G_P$, and $\chi(G_P)=\omega(G_P)$. That is, (D) = (E).

  \medskip\AllBold{(F) $\leq$ (C):} Let $P$ be $n$ evenly spaced
  points on a circle.  Let $q$ be the point near the centre of the
  circle defined in the proof that (E) $\leq$ (D). Let $X$ be the set
  of triangles determined by $P$ that contain $q$ in their
  interior. Thus (C) $\geq |X|$. Let $G$ be the $n$-vertex tournament
  with vertex set $P$, where the edge $vw$ is directed from $v$ to $w$
  whenever $w$ is clockwise from $v$ in the triangle $vwq$. Observe
  that if $n$ is odd, then every vertex in $G$ has outdegree
  $\frac{n-1}{2}$. And if $n$ is even, then half the vertices in $G$
  have outdegree $\frac{n-2}{2}$ and the other half have outdegree
  $\frac{n}{2}$. The analysis in the proof that (A) $\leq$ (F) shows that $|X|=$ (F). Hence (C) $\geq$ (F).

  \medskip\AllBold{(E) $\leq$ (A):} Let $P$ be $n$ evenly spaced
  points on a circle.  Let $q$ be the point near the centre of the
  circle defined in the proof that (E) $\leq$ (D). Let $G$ be the
  $n$-vertex tournament with vertex set $P$, where the edge $vw$ is
  directed from $v$ to $w$ whenever $w$ is clockwise from $v$ in the
  triangle $vwq$. Three vertices form a directed 3-cycle in $G$ if and
  only if they form a central triangle. Thus (A) is at least the
  number of central triangles, which equals (E) by the proof that (E)
  $\leq$ (D).

  \medskip We have proved that (F) $\leq$ (C) $\leq$ (B) $\leq$ (A)
  $\leq$ (F) and (F) $\leq$ (C) $\leq$ (D) $\leq$ (E) $\leq$ (A)
  $\leq$ (F).  Thus (A) = (B) = (C) = (D) = (E) = (F).
\end{proof}

% \begin{lemma}
% For all $n$ there is a point set $P$ in convex position and a point
% $q$, such that if $X$ is the set of triangles determined by $P$ that
% contain $q$ in their interior, then 
%    \begin{equation*}
% \chi(G_P)=\omega(G_P)=
%       \begin{cases}
%         \tfrac{1}{24}\,(n-1)n(n+1)&\text{ if $n$ is odd}\\
%         \tfrac{1}{24}\,(n-2)n(n+2)&\text{ if $n$ is even}\enspace.
%       \end{cases}
%     \end{equation*}
% \end{lemma}

% \begin{proof}

% \end{proof}

We conjecture that the maximum clique number of the intersection graph
of the open triangles determined by $n$ points in general position
also equals the number in Theorem~\ref{thm:Related}, as does the
maximum chromatic number. It may even be true that
$\chi(G_P)=\omega(G_P)$ for every set $P$ of points in general
position. We have verified by computer that $\chi(G_P)=\omega(G_P)$ for every set
$P$ of at most 7 points in general position.

\newpage
\section{8-Colouring the Triangles Determined by 6 Points}
\label{SixPoints}
\definecolor{MidRed}{rgb}{1,0.4,0.4}
\definecolor{LightBlue}{rgb}{0.8,0.8,1}

\psset{unit=0.11mm}

\newcommand{\NewOrderType}{
\medskip

\noindent
\hrulefill

\medskip
\noindent}

\newcommand{\drawing}[1]{
\begin{minipage}{37mm}
\begin{pspicture}(330,300)
#1
\end{pspicture}
\end{minipage}}

\noindent
\drawing{
\psline[linecolor=lightgray,linewidth=1pt](30,160)(126,220)
\psline[linecolor=lightgray,linewidth=1pt](126,220)(230,192)
\psline[linecolor=lightgray,linewidth=1pt](230,192)(248,78)
\psline[linecolor=lightgray,linewidth=1pt](30,160)(54,72)
\psline[linecolor=lightgray,linewidth=1pt](54,72)(148,36)
\psline[linecolor=lightgray,linewidth=1pt](148,36)(248,78)
\pspolygon[fillstyle=solid,fillcolor=LightBlue,linecolor=black](148,36)(54,72)(126,220)
\pspolygon[fillstyle=solid,fillcolor=LightBlue,linecolor=black](248,78)(148,36)(126,220)
\pspolygon[fillstyle=solid,fillcolor=LightBlue,linecolor=black](54,72)(30,160)(126,220)
\pspolygon[fillstyle=solid,fillcolor=LightBlue,linecolor=black](230,192)(248,78)(126,220)
\pscircle[fillstyle=solid,fillcolor=MidRed,linecolor=black](30,160){6}
\pscircle[fillstyle=solid,fillcolor=MidRed,linecolor=black](54,72){6}
\pscircle[fillstyle=solid,fillcolor=MidRed,linecolor=black](126,220){6}
\pscircle[fillstyle=solid,fillcolor=MidRed,linecolor=black](148,36){6}
\pscircle[fillstyle=solid,fillcolor=MidRed,linecolor=black](230,192){6}
\pscircle[fillstyle=solid,fillcolor=MidRed,linecolor=black](248,78){6}
}
\drawing{
\psline[linecolor=lightgray,linewidth=1pt](30,160)(126,220)
\psline[linecolor=lightgray,linewidth=1pt](126,220)(230,192)
\psline[linecolor=lightgray,linewidth=1pt](230,192)(248,78)
\psline[linecolor=lightgray,linewidth=1pt](30,160)(54,72)
\psline[linecolor=lightgray,linewidth=1pt](54,72)(148,36)
\psline[linecolor=lightgray,linewidth=1pt](148,36)(248,78)
\pspolygon[fillstyle=solid,fillcolor=LightBlue,linecolor=black](248,78)(54,72)(126,220)
\pscircle[fillstyle=solid,fillcolor=MidRed,linecolor=black](30,160){6}
\pscircle[fillstyle=solid,fillcolor=MidRed,linecolor=black](54,72){6}
\pscircle[fillstyle=solid,fillcolor=MidRed,linecolor=black](126,220){6}
\pscircle[fillstyle=solid,fillcolor=MidRed,linecolor=black](148,36){6}
\pscircle[fillstyle=solid,fillcolor=MidRed,linecolor=black](230,192){6}
\pscircle[fillstyle=solid,fillcolor=MidRed,linecolor=black](248,78){6}
}
\drawing{
\psline[linecolor=lightgray,linewidth=1pt](30,160)(126,220)
\psline[linecolor=lightgray,linewidth=1pt](126,220)(230,192)
\psline[linecolor=lightgray,linewidth=1pt](230,192)(248,78)
\psline[linecolor=lightgray,linewidth=1pt](30,160)(54,72)
\psline[linecolor=lightgray,linewidth=1pt](54,72)(148,36)
\psline[linecolor=lightgray,linewidth=1pt](148,36)(248,78)
\pspolygon[fillstyle=solid,fillcolor=LightBlue,linecolor=black](248,78)(30,160)(126,220)
\pspolygon[fillstyle=solid,fillcolor=LightBlue,linecolor=black](248,78)(54,72)(30,160)
\pscircle[fillstyle=solid,fillcolor=MidRed,linecolor=black](30,160){6}
\pscircle[fillstyle=solid,fillcolor=MidRed,linecolor=black](54,72){6}
\pscircle[fillstyle=solid,fillcolor=MidRed,linecolor=black](126,220){6}
\pscircle[fillstyle=solid,fillcolor=MidRed,linecolor=black](148,36){6}
\pscircle[fillstyle=solid,fillcolor=MidRed,linecolor=black](230,192){6}
\pscircle[fillstyle=solid,fillcolor=MidRed,linecolor=black](248,78){6}
}
\drawing{
\psline[linecolor=lightgray,linewidth=1pt](30,160)(126,220)
\psline[linecolor=lightgray,linewidth=1pt](126,220)(230,192)
\psline[linecolor=lightgray,linewidth=1pt](230,192)(248,78)
\psline[linecolor=lightgray,linewidth=1pt](30,160)(54,72)
\psline[linecolor=lightgray,linewidth=1pt](54,72)(148,36)
\psline[linecolor=lightgray,linewidth=1pt](148,36)(248,78)
\pspolygon[fillstyle=solid,fillcolor=LightBlue,linecolor=black](230,192)(248,78)(30,160)
\pspolygon[fillstyle=solid,fillcolor=LightBlue,linecolor=black](248,78)(148,36)(30,160)
\pscircle[fillstyle=solid,fillcolor=MidRed,linecolor=black](30,160){6}
\pscircle[fillstyle=solid,fillcolor=MidRed,linecolor=black](54,72){6}
\pscircle[fillstyle=solid,fillcolor=MidRed,linecolor=black](126,220){6}
\pscircle[fillstyle=solid,fillcolor=MidRed,linecolor=black](148,36){6}
\pscircle[fillstyle=solid,fillcolor=MidRed,linecolor=black](230,192){6}
\pscircle[fillstyle=solid,fillcolor=MidRed,linecolor=black](248,78){6}
}
\drawing{
\psline[linecolor=lightgray,linewidth=1pt](30,160)(126,220)
\psline[linecolor=lightgray,linewidth=1pt](126,220)(230,192)
\psline[linecolor=lightgray,linewidth=1pt](230,192)(248,78)
\psline[linecolor=lightgray,linewidth=1pt](30,160)(54,72)
\psline[linecolor=lightgray,linewidth=1pt](54,72)(148,36)
\psline[linecolor=lightgray,linewidth=1pt](148,36)(248,78)
\pspolygon[fillstyle=solid,fillcolor=LightBlue,linecolor=black](230,192)(148,36)(54,72)
\pspolygon[fillstyle=solid,fillcolor=LightBlue,linecolor=black](230,192)(54,72)(126,220)
\pscircle[fillstyle=solid,fillcolor=MidRed,linecolor=black](30,160){6}
\pscircle[fillstyle=solid,fillcolor=MidRed,linecolor=black](54,72){6}
\pscircle[fillstyle=solid,fillcolor=MidRed,linecolor=black](126,220){6}
\pscircle[fillstyle=solid,fillcolor=MidRed,linecolor=black](148,36){6}
\pscircle[fillstyle=solid,fillcolor=MidRed,linecolor=black](230,192){6}
\pscircle[fillstyle=solid,fillcolor=MidRed,linecolor=black](248,78){6}
}
\drawing{
\psline[linecolor=lightgray,linewidth=1pt](30,160)(126,220)
\psline[linecolor=lightgray,linewidth=1pt](126,220)(230,192)
\psline[linecolor=lightgray,linewidth=1pt](230,192)(248,78)
\psline[linecolor=lightgray,linewidth=1pt](30,160)(54,72)
\psline[linecolor=lightgray,linewidth=1pt](54,72)(148,36)
\psline[linecolor=lightgray,linewidth=1pt](148,36)(248,78)
\pspolygon[fillstyle=solid,fillcolor=LightBlue,linecolor=black](230,192)(248,78)(54,72)
\pspolygon[fillstyle=solid,fillcolor=LightBlue,linecolor=black](230,192)(54,72)(30,160)
\pspolygon[fillstyle=solid,fillcolor=LightBlue,linecolor=black](248,78)(148,36)(54,72)
\pscircle[fillstyle=solid,fillcolor=MidRed,linecolor=black](30,160){6}
\pscircle[fillstyle=solid,fillcolor=MidRed,linecolor=black](54,72){6}
\pscircle[fillstyle=solid,fillcolor=MidRed,linecolor=black](126,220){6}
\pscircle[fillstyle=solid,fillcolor=MidRed,linecolor=black](148,36){6}
\pscircle[fillstyle=solid,fillcolor=MidRed,linecolor=black](230,192){6}
\pscircle[fillstyle=solid,fillcolor=MidRed,linecolor=black](248,78){6}
}
\drawing{
\psline[linecolor=lightgray,linewidth=1pt](30,160)(126,220)
\psline[linecolor=lightgray,linewidth=1pt](126,220)(230,192)
\psline[linecolor=lightgray,linewidth=1pt](230,192)(248,78)
\psline[linecolor=lightgray,linewidth=1pt](30,160)(54,72)
\psline[linecolor=lightgray,linewidth=1pt](54,72)(148,36)
\psline[linecolor=lightgray,linewidth=1pt](148,36)(248,78)
\pspolygon[fillstyle=solid,fillcolor=LightBlue,linecolor=black](230,192)(148,36)(30,160)
\pspolygon[fillstyle=solid,fillcolor=LightBlue,linecolor=black](230,192)(30,160)(126,220)
\pscircle[fillstyle=solid,fillcolor=MidRed,linecolor=black](30,160){6}
\pscircle[fillstyle=solid,fillcolor=MidRed,linecolor=black](54,72){6}
\pscircle[fillstyle=solid,fillcolor=MidRed,linecolor=black](126,220){6}
\pscircle[fillstyle=solid,fillcolor=MidRed,linecolor=black](148,36){6}
\pscircle[fillstyle=solid,fillcolor=MidRed,linecolor=black](230,192){6}
\pscircle[fillstyle=solid,fillcolor=MidRed,linecolor=black](248,78){6}
}
\drawing{
\psline[linecolor=lightgray,linewidth=1pt](30,160)(126,220)
\psline[linecolor=lightgray,linewidth=1pt](126,220)(230,192)
\psline[linecolor=lightgray,linewidth=1pt](230,192)(248,78)
\psline[linecolor=lightgray,linewidth=1pt](30,160)(54,72)
\psline[linecolor=lightgray,linewidth=1pt](54,72)(148,36)
\psline[linecolor=lightgray,linewidth=1pt](148,36)(248,78)
\pspolygon[fillstyle=solid,fillcolor=LightBlue,linecolor=black](148,36)(30,160)(126,220)
\pspolygon[fillstyle=solid,fillcolor=LightBlue,linecolor=black](230,192)(148,36)(126,220)
\pspolygon[fillstyle=solid,fillcolor=LightBlue,linecolor=black](148,36)(54,72)(30,160)
\pspolygon[fillstyle=solid,fillcolor=LightBlue,linecolor=black](230,192)(248,78)(148,36)
\pscircle[fillstyle=solid,fillcolor=MidRed,linecolor=black](30,160){6}
\pscircle[fillstyle=solid,fillcolor=MidRed,linecolor=black](54,72){6}
\pscircle[fillstyle=solid,fillcolor=MidRed,linecolor=black](126,220){6}
\pscircle[fillstyle=solid,fillcolor=MidRed,linecolor=black](148,36){6}
\pscircle[fillstyle=solid,fillcolor=MidRed,linecolor=black](230,192){6}
\pscircle[fillstyle=solid,fillcolor=MidRed,linecolor=black](248,78){6}
}
\NewOrderType
\drawing{
\psline[linecolor=lightgray,linewidth=1pt](11,37)(76,255)
\psline[linecolor=lightgray,linewidth=1pt](76,255)(244,198)
\psline[linecolor=lightgray,linewidth=1pt](11,37)(84,4)
\psline[linecolor=lightgray,linewidth=1pt](84,4)(169,32)
\psline[linecolor=lightgray,linewidth=1pt](169,32)(244,198)
\pspolygon[fillstyle=solid,fillcolor=LightBlue,linecolor=black](11,37)(76,255)(244,198)
\pspolygon[fillstyle=solid,fillcolor=LightBlue,linecolor=black](84,4)(11,37)(244,198)
\pspolygon[fillstyle=solid,fillcolor=LightBlue,linecolor=black](169,32)(84,4)(244,198)
\pscircle[fillstyle=solid,fillcolor=MidRed,linecolor=black](11,37){6}
\pscircle[fillstyle=solid,fillcolor=MidRed,linecolor=black](76,255){6}
\pscircle[fillstyle=solid,fillcolor=MidRed,linecolor=black](84,4){6}
\pscircle[fillstyle=solid,fillcolor=MidRed,linecolor=black](147,205){6}
\pscircle[fillstyle=solid,fillcolor=MidRed,linecolor=black](169,32){6}
\pscircle[fillstyle=solid,fillcolor=MidRed,linecolor=black](244,198){6}
}
\drawing{
\psline[linecolor=lightgray,linewidth=1pt](11,37)(76,255)
\psline[linecolor=lightgray,linewidth=1pt](76,255)(244,198)
\psline[linecolor=lightgray,linewidth=1pt](11,37)(84,4)
\psline[linecolor=lightgray,linewidth=1pt](84,4)(169,32)
\psline[linecolor=lightgray,linewidth=1pt](169,32)(244,198)
\pspolygon[fillstyle=solid,fillcolor=LightBlue,linecolor=black](84,4)(76,255)(244,198)
\pscircle[fillstyle=solid,fillcolor=MidRed,linecolor=black](11,37){6}
\pscircle[fillstyle=solid,fillcolor=MidRed,linecolor=black](76,255){6}
\pscircle[fillstyle=solid,fillcolor=MidRed,linecolor=black](84,4){6}
\pscircle[fillstyle=solid,fillcolor=MidRed,linecolor=black](147,205){6}
\pscircle[fillstyle=solid,fillcolor=MidRed,linecolor=black](169,32){6}
\pscircle[fillstyle=solid,fillcolor=MidRed,linecolor=black](244,198){6}
}
\drawing{
\psline[linecolor=lightgray,linewidth=1pt](11,37)(76,255)
\psline[linecolor=lightgray,linewidth=1pt](76,255)(244,198)
\psline[linecolor=lightgray,linewidth=1pt](11,37)(84,4)
\psline[linecolor=lightgray,linewidth=1pt](84,4)(169,32)
\psline[linecolor=lightgray,linewidth=1pt](169,32)(244,198)
\pspolygon[fillstyle=solid,fillcolor=LightBlue,linecolor=black](84,4)(147,205)(244,198)
\pspolygon[fillstyle=solid,fillcolor=LightBlue,linecolor=black](84,4)(76,255)(147,205)
\pscircle[fillstyle=solid,fillcolor=MidRed,linecolor=black](11,37){6}
\pscircle[fillstyle=solid,fillcolor=MidRed,linecolor=black](76,255){6}
\pscircle[fillstyle=solid,fillcolor=MidRed,linecolor=black](84,4){6}
\pscircle[fillstyle=solid,fillcolor=MidRed,linecolor=black](147,205){6}
\pscircle[fillstyle=solid,fillcolor=MidRed,linecolor=black](169,32){6}
\pscircle[fillstyle=solid,fillcolor=MidRed,linecolor=black](244,198){6}
}
\drawing{
\psline[linecolor=lightgray,linewidth=1pt](11,37)(76,255)
\psline[linecolor=lightgray,linewidth=1pt](76,255)(244,198)
\psline[linecolor=lightgray,linewidth=1pt](11,37)(84,4)
\psline[linecolor=lightgray,linewidth=1pt](84,4)(169,32)
\psline[linecolor=lightgray,linewidth=1pt](169,32)(244,198)
\pspolygon[fillstyle=solid,fillcolor=LightBlue,linecolor=black](169,32)(76,255)(244,198)
\pspolygon[fillstyle=solid,fillcolor=LightBlue,linecolor=black](169,32)(84,4)(76,255)
\pspolygon[fillstyle=solid,fillcolor=LightBlue,linecolor=black](84,4)(11,37)(76,255)
\pscircle[fillstyle=solid,fillcolor=MidRed,linecolor=black](11,37){6}
\pscircle[fillstyle=solid,fillcolor=MidRed,linecolor=black](76,255){6}
\pscircle[fillstyle=solid,fillcolor=MidRed,linecolor=black](84,4){6}
\pscircle[fillstyle=solid,fillcolor=MidRed,linecolor=black](147,205){6}
\pscircle[fillstyle=solid,fillcolor=MidRed,linecolor=black](169,32){6}
\pscircle[fillstyle=solid,fillcolor=MidRed,linecolor=black](244,198){6}
}
\drawing{
\psline[linecolor=lightgray,linewidth=1pt](11,37)(76,255)
\psline[linecolor=lightgray,linewidth=1pt](76,255)(244,198)
\psline[linecolor=lightgray,linewidth=1pt](11,37)(84,4)
\psline[linecolor=lightgray,linewidth=1pt](84,4)(169,32)
\psline[linecolor=lightgray,linewidth=1pt](169,32)(244,198)
\pspolygon[fillstyle=solid,fillcolor=LightBlue,linecolor=black](169,32)(11,37)(76,255)
\pspolygon[fillstyle=solid,fillcolor=LightBlue,linecolor=black](169,32)(76,255)(147,205)
\pscircle[fillstyle=solid,fillcolor=MidRed,linecolor=black](11,37){6}
\pscircle[fillstyle=solid,fillcolor=MidRed,linecolor=black](76,255){6}
\pscircle[fillstyle=solid,fillcolor=MidRed,linecolor=black](84,4){6}
\pscircle[fillstyle=solid,fillcolor=MidRed,linecolor=black](147,205){6}
\pscircle[fillstyle=solid,fillcolor=MidRed,linecolor=black](169,32){6}
\pscircle[fillstyle=solid,fillcolor=MidRed,linecolor=black](244,198){6}
}
\drawing{
\psline[linecolor=lightgray,linewidth=1pt](11,37)(76,255)
\psline[linecolor=lightgray,linewidth=1pt](76,255)(244,198)
\psline[linecolor=lightgray,linewidth=1pt](11,37)(84,4)
\psline[linecolor=lightgray,linewidth=1pt](84,4)(169,32)
\psline[linecolor=lightgray,linewidth=1pt](169,32)(244,198)
\pspolygon[fillstyle=solid,fillcolor=LightBlue,linecolor=black](169,32)(84,4)(147,205)
\pspolygon[fillstyle=solid,fillcolor=LightBlue,linecolor=black](84,4)(11,37)(147,205)
\pscircle[fillstyle=solid,fillcolor=MidRed,linecolor=black](11,37){6}
\pscircle[fillstyle=solid,fillcolor=MidRed,linecolor=black](76,255){6}
\pscircle[fillstyle=solid,fillcolor=MidRed,linecolor=black](84,4){6}
\pscircle[fillstyle=solid,fillcolor=MidRed,linecolor=black](147,205){6}
\pscircle[fillstyle=solid,fillcolor=MidRed,linecolor=black](169,32){6}
\pscircle[fillstyle=solid,fillcolor=MidRed,linecolor=black](244,198){6}
}
\drawing{
\psline[linecolor=lightgray,linewidth=1pt](11,37)(76,255)
\psline[linecolor=lightgray,linewidth=1pt](76,255)(244,198)
\psline[linecolor=lightgray,linewidth=1pt](11,37)(84,4)
\psline[linecolor=lightgray,linewidth=1pt](84,4)(169,32)
\psline[linecolor=lightgray,linewidth=1pt](169,32)(244,198)
\pspolygon[fillstyle=solid,fillcolor=LightBlue,linecolor=black](169,32)(11,37)(147,205)
\pspolygon[fillstyle=solid,fillcolor=LightBlue,linecolor=black](169,32)(147,205)(244,198)
\pscircle[fillstyle=solid,fillcolor=MidRed,linecolor=black](11,37){6}
\pscircle[fillstyle=solid,fillcolor=MidRed,linecolor=black](76,255){6}
\pscircle[fillstyle=solid,fillcolor=MidRed,linecolor=black](84,4){6}
\pscircle[fillstyle=solid,fillcolor=MidRed,linecolor=black](147,205){6}
\pscircle[fillstyle=solid,fillcolor=MidRed,linecolor=black](169,32){6}
\pscircle[fillstyle=solid,fillcolor=MidRed,linecolor=black](244,198){6}
}
\drawing{
\psline[linecolor=lightgray,linewidth=1pt](11,37)(76,255)
\psline[linecolor=lightgray,linewidth=1pt](76,255)(244,198)
\psline[linecolor=lightgray,linewidth=1pt](11,37)(84,4)
\psline[linecolor=lightgray,linewidth=1pt](84,4)(169,32)
\psline[linecolor=lightgray,linewidth=1pt](169,32)(244,198)
\pspolygon[fillstyle=solid,fillcolor=LightBlue,linecolor=black](11,37)(147,205)(244,198)
\pspolygon[fillstyle=solid,fillcolor=LightBlue,linecolor=black](169,32)(11,37)(244,198)
\pspolygon[fillstyle=solid,fillcolor=LightBlue,linecolor=black](11,37)(76,255)(147,205)
\pspolygon[fillstyle=solid,fillcolor=LightBlue,linecolor=black](76,255)(147,205)(244,198)
\pspolygon[fillstyle=solid,fillcolor=LightBlue,linecolor=black](169,32)(84,4)(11,37)
\pscircle[fillstyle=solid,fillcolor=MidRed,linecolor=black](11,37){6}
\pscircle[fillstyle=solid,fillcolor=MidRed,linecolor=black](76,255){6}
\pscircle[fillstyle=solid,fillcolor=MidRed,linecolor=black](84,4){6}
\pscircle[fillstyle=solid,fillcolor=MidRed,linecolor=black](147,205){6}
\pscircle[fillstyle=solid,fillcolor=MidRed,linecolor=black](169,32){6}
\pscircle[fillstyle=solid,fillcolor=MidRed,linecolor=black](244,198){6}
}
\NewOrderType
\drawing{
\psline[linecolor=lightgray,linewidth=1pt](7,145)(113,247)
\psline[linecolor=lightgray,linewidth=1pt](113,247)(177,236)
\psline[linecolor=lightgray,linewidth=1pt](177,236)(249,57)
\psline[linecolor=lightgray,linewidth=1pt](7,145)(208,8)
\psline[linecolor=lightgray,linewidth=1pt](208,8)(249,57)
\pspolygon[fillstyle=solid,fillcolor=LightBlue,linecolor=black](113,247)(177,236)(249,57)
\pspolygon[fillstyle=solid,fillcolor=LightBlue,linecolor=black](113,247)(249,57)(208,8)
\pspolygon[fillstyle=solid,fillcolor=LightBlue,linecolor=black](113,247)(208,8)(93,161)
\pspolygon[fillstyle=solid,fillcolor=LightBlue,linecolor=black](113,247)(93,161)(7,145)
\pspolygon[fillstyle=solid,fillcolor=LightBlue,linecolor=black](208,8)(93,161)(7,145)
\pscircle[fillstyle=solid,fillcolor=MidRed,linecolor=black](7,145){6}
\pscircle[fillstyle=solid,fillcolor=MidRed,linecolor=black](93,161){6}
\pscircle[fillstyle=solid,fillcolor=MidRed,linecolor=black](113,247){6}
\pscircle[fillstyle=solid,fillcolor=MidRed,linecolor=black](177,236){6}
\pscircle[fillstyle=solid,fillcolor=MidRed,linecolor=black](208,8){6}
\pscircle[fillstyle=solid,fillcolor=MidRed,linecolor=black](249,57){6}
}
\drawing{
\psline[linecolor=lightgray,linewidth=1pt](7,145)(113,247)
\psline[linecolor=lightgray,linewidth=1pt](113,247)(177,236)
\psline[linecolor=lightgray,linewidth=1pt](177,236)(249,57)
\psline[linecolor=lightgray,linewidth=1pt](7,145)(208,8)
\psline[linecolor=lightgray,linewidth=1pt](208,8)(249,57)
\pspolygon[fillstyle=solid,fillcolor=LightBlue,linecolor=black](177,236)(93,161)(7,145)
\pspolygon[fillstyle=solid,fillcolor=LightBlue,linecolor=black](177,236)(249,57)(93,161)
\pspolygon[fillstyle=solid,fillcolor=LightBlue,linecolor=black](249,57)(93,161)(7,145)
\pscircle[fillstyle=solid,fillcolor=MidRed,linecolor=black](7,145){6}
\pscircle[fillstyle=solid,fillcolor=MidRed,linecolor=black](93,161){6}
\pscircle[fillstyle=solid,fillcolor=MidRed,linecolor=black](113,247){6}
\pscircle[fillstyle=solid,fillcolor=MidRed,linecolor=black](177,236){6}
\pscircle[fillstyle=solid,fillcolor=MidRed,linecolor=black](208,8){6}
\pscircle[fillstyle=solid,fillcolor=MidRed,linecolor=black](249,57){6}
}
\drawing{
\psline[linecolor=lightgray,linewidth=1pt](7,145)(113,247)
\psline[linecolor=lightgray,linewidth=1pt](113,247)(177,236)
\psline[linecolor=lightgray,linewidth=1pt](177,236)(249,57)
\psline[linecolor=lightgray,linewidth=1pt](7,145)(208,8)
\psline[linecolor=lightgray,linewidth=1pt](208,8)(249,57)
\pspolygon[fillstyle=solid,fillcolor=LightBlue,linecolor=black](177,236)(249,57)(7,145)
\pscircle[fillstyle=solid,fillcolor=MidRed,linecolor=black](7,145){6}
\pscircle[fillstyle=solid,fillcolor=MidRed,linecolor=black](93,161){6}
\pscircle[fillstyle=solid,fillcolor=MidRed,linecolor=black](113,247){6}
\pscircle[fillstyle=solid,fillcolor=MidRed,linecolor=black](177,236){6}
\pscircle[fillstyle=solid,fillcolor=MidRed,linecolor=black](208,8){6}
\pscircle[fillstyle=solid,fillcolor=MidRed,linecolor=black](249,57){6}
}
\drawing{
\psline[linecolor=lightgray,linewidth=1pt](7,145)(113,247)
\psline[linecolor=lightgray,linewidth=1pt](113,247)(177,236)
\psline[linecolor=lightgray,linewidth=1pt](177,236)(249,57)
\psline[linecolor=lightgray,linewidth=1pt](7,145)(208,8)
\psline[linecolor=lightgray,linewidth=1pt](208,8)(249,57)
\pspolygon[fillstyle=solid,fillcolor=LightBlue,linecolor=black](113,247)(249,57)(7,145)
\pspolygon[fillstyle=solid,fillcolor=LightBlue,linecolor=black](249,57)(208,8)(7,145)
\pscircle[fillstyle=solid,fillcolor=MidRed,linecolor=black](7,145){6}
\pscircle[fillstyle=solid,fillcolor=MidRed,linecolor=black](93,161){6}
\pscircle[fillstyle=solid,fillcolor=MidRed,linecolor=black](113,247){6}
\pscircle[fillstyle=solid,fillcolor=MidRed,linecolor=black](177,236){6}
\pscircle[fillstyle=solid,fillcolor=MidRed,linecolor=black](208,8){6}
\pscircle[fillstyle=solid,fillcolor=MidRed,linecolor=black](249,57){6}
}
\drawing{
\psline[linecolor=lightgray,linewidth=1pt](7,145)(113,247)
\psline[linecolor=lightgray,linewidth=1pt](113,247)(177,236)
\psline[linecolor=lightgray,linewidth=1pt](177,236)(249,57)
\psline[linecolor=lightgray,linewidth=1pt](7,145)(208,8)
\psline[linecolor=lightgray,linewidth=1pt](208,8)(249,57)
\pspolygon[fillstyle=solid,fillcolor=LightBlue,linecolor=black](113,247)(177,236)(93,161)
\pspolygon[fillstyle=solid,fillcolor=LightBlue,linecolor=black](177,236)(208,8)(93,161)
\pscircle[fillstyle=solid,fillcolor=MidRed,linecolor=black](7,145){6}
\pscircle[fillstyle=solid,fillcolor=MidRed,linecolor=black](93,161){6}
\pscircle[fillstyle=solid,fillcolor=MidRed,linecolor=black](113,247){6}
\pscircle[fillstyle=solid,fillcolor=MidRed,linecolor=black](177,236){6}
\pscircle[fillstyle=solid,fillcolor=MidRed,linecolor=black](208,8){6}
\pscircle[fillstyle=solid,fillcolor=MidRed,linecolor=black](249,57){6}
}
\drawing{
\psline[linecolor=lightgray,linewidth=1pt](7,145)(113,247)
\psline[linecolor=lightgray,linewidth=1pt](113,247)(177,236)
\psline[linecolor=lightgray,linewidth=1pt](177,236)(249,57)
\psline[linecolor=lightgray,linewidth=1pt](7,145)(208,8)
\psline[linecolor=lightgray,linewidth=1pt](208,8)(249,57)
\pspolygon[fillstyle=solid,fillcolor=LightBlue,linecolor=black](113,247)(249,57)(93,161)
\pspolygon[fillstyle=solid,fillcolor=LightBlue,linecolor=black](249,57)(208,8)(93,161)
\pscircle[fillstyle=solid,fillcolor=MidRed,linecolor=black](7,145){6}
\pscircle[fillstyle=solid,fillcolor=MidRed,linecolor=black](93,161){6}
\pscircle[fillstyle=solid,fillcolor=MidRed,linecolor=black](113,247){6}
\pscircle[fillstyle=solid,fillcolor=MidRed,linecolor=black](177,236){6}
\pscircle[fillstyle=solid,fillcolor=MidRed,linecolor=black](208,8){6}
\pscircle[fillstyle=solid,fillcolor=MidRed,linecolor=black](249,57){6}
}
\drawing{
\psline[linecolor=lightgray,linewidth=1pt](7,145)(113,247)
\psline[linecolor=lightgray,linewidth=1pt](113,247)(177,236)
\psline[linecolor=lightgray,linewidth=1pt](177,236)(249,57)
\psline[linecolor=lightgray,linewidth=1pt](7,145)(208,8)
\psline[linecolor=lightgray,linewidth=1pt](208,8)(249,57)
\pspolygon[fillstyle=solid,fillcolor=LightBlue,linecolor=black](113,247)(177,236)(7,145)
\pspolygon[fillstyle=solid,fillcolor=LightBlue,linecolor=black](177,236)(208,8)(7,145)
\pscircle[fillstyle=solid,fillcolor=MidRed,linecolor=black](7,145){6}
\pscircle[fillstyle=solid,fillcolor=MidRed,linecolor=black](93,161){6}
\pscircle[fillstyle=solid,fillcolor=MidRed,linecolor=black](113,247){6}
\pscircle[fillstyle=solid,fillcolor=MidRed,linecolor=black](177,236){6}
\pscircle[fillstyle=solid,fillcolor=MidRed,linecolor=black](208,8){6}
\pscircle[fillstyle=solid,fillcolor=MidRed,linecolor=black](249,57){6}
}
\drawing{
\psline[linecolor=lightgray,linewidth=1pt](7,145)(113,247)
\psline[linecolor=lightgray,linewidth=1pt](113,247)(177,236)
\psline[linecolor=lightgray,linewidth=1pt](177,236)(249,57)
\psline[linecolor=lightgray,linewidth=1pt](7,145)(208,8)
\psline[linecolor=lightgray,linewidth=1pt](208,8)(249,57)
\pspolygon[fillstyle=solid,fillcolor=LightBlue,linecolor=black](113,247)(177,236)(208,8)
\pspolygon[fillstyle=solid,fillcolor=LightBlue,linecolor=black](113,247)(208,8)(7,145)
\pspolygon[fillstyle=solid,fillcolor=LightBlue,linecolor=black](177,236)(249,57)(208,8)
\pscircle[fillstyle=solid,fillcolor=MidRed,linecolor=black](7,145){6}
\pscircle[fillstyle=solid,fillcolor=MidRed,linecolor=black](93,161){6}
\pscircle[fillstyle=solid,fillcolor=MidRed,linecolor=black](113,247){6}
\pscircle[fillstyle=solid,fillcolor=MidRed,linecolor=black](177,236){6}
\pscircle[fillstyle=solid,fillcolor=MidRed,linecolor=black](208,8){6}
\pscircle[fillstyle=solid,fillcolor=MidRed,linecolor=black](249,57){6}
}
\NewOrderType
\drawing{
\psline[linecolor=lightgray,linewidth=1pt](0,172)(16,198)
\psline[linecolor=lightgray,linewidth=1pt](16,198)(241,244)
\psline[linecolor=lightgray,linewidth=1pt](0,172)(99,12)
\psline[linecolor=lightgray,linewidth=1pt](99,12)(241,244)
\pspolygon[fillstyle=solid,fillcolor=LightBlue,linecolor=black](0,172)(157,140)(99,12)
\pspolygon[fillstyle=solid,fillcolor=LightBlue,linecolor=black](0,172)(241,244)(157,140)
\pspolygon[fillstyle=solid,fillcolor=LightBlue,linecolor=black](0,172)(16,198)(241,244)
\pscircle[fillstyle=solid,fillcolor=MidRed,linecolor=black](0,172){6}
\pscircle[fillstyle=solid,fillcolor=MidRed,linecolor=black](16,198){6}
\pscircle[fillstyle=solid,fillcolor=MidRed,linecolor=black](99,12){6}
\pscircle[fillstyle=solid,fillcolor=MidRed,linecolor=black](101,110){6}
\pscircle[fillstyle=solid,fillcolor=MidRed,linecolor=black](157,140){6}
\pscircle[fillstyle=solid,fillcolor=MidRed,linecolor=black](241,244){6}
}
\drawing{
\psline[linecolor=lightgray,linewidth=1pt](0,172)(16,198)
\psline[linecolor=lightgray,linewidth=1pt](16,198)(241,244)
\psline[linecolor=lightgray,linewidth=1pt](0,172)(99,12)
\psline[linecolor=lightgray,linewidth=1pt](99,12)(241,244)
\pspolygon[fillstyle=solid,fillcolor=LightBlue,linecolor=black](0,172)(241,244)(101,110)
\pspolygon[fillstyle=solid,fillcolor=LightBlue,linecolor=black](241,244)(157,140)(101,110)
\pscircle[fillstyle=solid,fillcolor=MidRed,linecolor=black](0,172){6}
\pscircle[fillstyle=solid,fillcolor=MidRed,linecolor=black](16,198){6}
\pscircle[fillstyle=solid,fillcolor=MidRed,linecolor=black](99,12){6}
\pscircle[fillstyle=solid,fillcolor=MidRed,linecolor=black](101,110){6}
\pscircle[fillstyle=solid,fillcolor=MidRed,linecolor=black](157,140){6}
\pscircle[fillstyle=solid,fillcolor=MidRed,linecolor=black](241,244){6}
}
\drawing{
\psline[linecolor=lightgray,linewidth=1pt](0,172)(16,198)
\psline[linecolor=lightgray,linewidth=1pt](16,198)(241,244)
\psline[linecolor=lightgray,linewidth=1pt](0,172)(99,12)
\psline[linecolor=lightgray,linewidth=1pt](99,12)(241,244)
\pspolygon[fillstyle=solid,fillcolor=LightBlue,linecolor=black](0,172)(16,198)(101,110)
\pspolygon[fillstyle=solid,fillcolor=LightBlue,linecolor=black](16,198)(241,244)(101,110)
\pspolygon[fillstyle=solid,fillcolor=LightBlue,linecolor=black](241,244)(101,110)(99,12)
\pscircle[fillstyle=solid,fillcolor=MidRed,linecolor=black](0,172){6}
\pscircle[fillstyle=solid,fillcolor=MidRed,linecolor=black](16,198){6}
\pscircle[fillstyle=solid,fillcolor=MidRed,linecolor=black](99,12){6}
\pscircle[fillstyle=solid,fillcolor=MidRed,linecolor=black](101,110){6}
\pscircle[fillstyle=solid,fillcolor=MidRed,linecolor=black](157,140){6}
\pscircle[fillstyle=solid,fillcolor=MidRed,linecolor=black](241,244){6}
}
\drawing{
\psline[linecolor=lightgray,linewidth=1pt](0,172)(16,198)
\psline[linecolor=lightgray,linewidth=1pt](16,198)(241,244)
\psline[linecolor=lightgray,linewidth=1pt](0,172)(99,12)
\psline[linecolor=lightgray,linewidth=1pt](99,12)(241,244)
\pspolygon[fillstyle=solid,fillcolor=LightBlue,linecolor=black](16,198)(157,140)(99,12)
\pscircle[fillstyle=solid,fillcolor=MidRed,linecolor=black](0,172){6}
\pscircle[fillstyle=solid,fillcolor=MidRed,linecolor=black](16,198){6}
\pscircle[fillstyle=solid,fillcolor=MidRed,linecolor=black](99,12){6}
\pscircle[fillstyle=solid,fillcolor=MidRed,linecolor=black](101,110){6}
\pscircle[fillstyle=solid,fillcolor=MidRed,linecolor=black](157,140){6}
\pscircle[fillstyle=solid,fillcolor=MidRed,linecolor=black](241,244){6}
}
\drawing{
\psline[linecolor=lightgray,linewidth=1pt](0,172)(16,198)
\psline[linecolor=lightgray,linewidth=1pt](16,198)(241,244)
\psline[linecolor=lightgray,linewidth=1pt](0,172)(99,12)
\psline[linecolor=lightgray,linewidth=1pt](99,12)(241,244)
\pspolygon[fillstyle=solid,fillcolor=LightBlue,linecolor=black](0,172)(241,244)(99,12)
\pscircle[fillstyle=solid,fillcolor=MidRed,linecolor=black](0,172){6}
\pscircle[fillstyle=solid,fillcolor=MidRed,linecolor=black](16,198){6}
\pscircle[fillstyle=solid,fillcolor=MidRed,linecolor=black](99,12){6}
\pscircle[fillstyle=solid,fillcolor=MidRed,linecolor=black](101,110){6}
\pscircle[fillstyle=solid,fillcolor=MidRed,linecolor=black](157,140){6}
\pscircle[fillstyle=solid,fillcolor=MidRed,linecolor=black](241,244){6}
}
\drawing{
\psline[linecolor=lightgray,linewidth=1pt](0,172)(16,198)
\psline[linecolor=lightgray,linewidth=1pt](16,198)(241,244)
\psline[linecolor=lightgray,linewidth=1pt](0,172)(99,12)
\psline[linecolor=lightgray,linewidth=1pt](99,12)(241,244)
\pspolygon[fillstyle=solid,fillcolor=LightBlue,linecolor=black](16,198)(241,244)(99,12)
\pscircle[fillstyle=solid,fillcolor=MidRed,linecolor=black](0,172){6}
\pscircle[fillstyle=solid,fillcolor=MidRed,linecolor=black](16,198){6}
\pscircle[fillstyle=solid,fillcolor=MidRed,linecolor=black](99,12){6}
\pscircle[fillstyle=solid,fillcolor=MidRed,linecolor=black](101,110){6}
\pscircle[fillstyle=solid,fillcolor=MidRed,linecolor=black](157,140){6}
\pscircle[fillstyle=solid,fillcolor=MidRed,linecolor=black](241,244){6}
}
\drawing{
\psline[linecolor=lightgray,linewidth=1pt](0,172)(16,198)
\psline[linecolor=lightgray,linewidth=1pt](16,198)(241,244)
\psline[linecolor=lightgray,linewidth=1pt](0,172)(99,12)
\psline[linecolor=lightgray,linewidth=1pt](99,12)(241,244)
\pspolygon[fillstyle=solid,fillcolor=LightBlue,linecolor=black](0,172)(16,198)(99,12)
\pspolygon[fillstyle=solid,fillcolor=LightBlue,linecolor=black](16,198)(101,110)(99,12)
\pspolygon[fillstyle=solid,fillcolor=LightBlue,linecolor=black](16,198)(157,140)(101,110)
\pscircle[fillstyle=solid,fillcolor=MidRed,linecolor=black](0,172){6}
\pscircle[fillstyle=solid,fillcolor=MidRed,linecolor=black](16,198){6}
\pscircle[fillstyle=solid,fillcolor=MidRed,linecolor=black](99,12){6}
\pscircle[fillstyle=solid,fillcolor=MidRed,linecolor=black](101,110){6}
\pscircle[fillstyle=solid,fillcolor=MidRed,linecolor=black](157,140){6}
\pscircle[fillstyle=solid,fillcolor=MidRed,linecolor=black](241,244){6}
}
\drawing{
\psline[linecolor=lightgray,linewidth=1pt](0,172)(16,198)
\psline[linecolor=lightgray,linewidth=1pt](16,198)(241,244)
\psline[linecolor=lightgray,linewidth=1pt](0,172)(99,12)
\psline[linecolor=lightgray,linewidth=1pt](99,12)(241,244)
\pspolygon[fillstyle=solid,fillcolor=LightBlue,linecolor=black](0,172)(101,110)(99,12)
\pspolygon[fillstyle=solid,fillcolor=LightBlue,linecolor=black](0,172)(157,140)(101,110)
\pspolygon[fillstyle=solid,fillcolor=LightBlue,linecolor=black](0,172)(16,198)(157,140)
\pspolygon[fillstyle=solid,fillcolor=LightBlue,linecolor=black](157,140)(101,110)(99,12)
\pspolygon[fillstyle=solid,fillcolor=LightBlue,linecolor=black](16,198)(241,244)(157,140)
\pspolygon[fillstyle=solid,fillcolor=LightBlue,linecolor=black](241,244)(157,140)(99,12)
\pscircle[fillstyle=solid,fillcolor=MidRed,linecolor=black](0,172){6}
\pscircle[fillstyle=solid,fillcolor=MidRed,linecolor=black](16,198){6}
\pscircle[fillstyle=solid,fillcolor=MidRed,linecolor=black](99,12){6}
\pscircle[fillstyle=solid,fillcolor=MidRed,linecolor=black](101,110){6}
\pscircle[fillstyle=solid,fillcolor=MidRed,linecolor=black](157,140){6}
\pscircle[fillstyle=solid,fillcolor=MidRed,linecolor=black](241,244){6}
}
\NewOrderType
\drawing{
\psline[linecolor=lightgray,linewidth=1pt](28,249)(248,120)
\psline[linecolor=lightgray,linewidth=1pt](28,249)(55,7)
\psline[linecolor=lightgray,linewidth=1pt](55,7)(91,10)
\psline[linecolor=lightgray,linewidth=1pt](91,10)(248,120)
\pspolygon[fillstyle=solid,fillcolor=LightBlue,linecolor=black](28,249)(62,189)(248,120)
\pspolygon[fillstyle=solid,fillcolor=LightBlue,linecolor=black](55,7)(62,189)(248,120)
\pscircle[fillstyle=solid,fillcolor=MidRed,linecolor=black](28,249){6}
\pscircle[fillstyle=solid,fillcolor=MidRed,linecolor=black](55,7){6}
\pscircle[fillstyle=solid,fillcolor=MidRed,linecolor=black](62,189){6}
\pscircle[fillstyle=solid,fillcolor=MidRed,linecolor=black](91,10){6}
\pscircle[fillstyle=solid,fillcolor=MidRed,linecolor=black](182,108){6}
\pscircle[fillstyle=solid,fillcolor=MidRed,linecolor=black](248,120){6}
}
\drawing{
\psline[linecolor=lightgray,linewidth=1pt](28,249)(248,120)
\psline[linecolor=lightgray,linewidth=1pt](28,249)(55,7)
\psline[linecolor=lightgray,linewidth=1pt](55,7)(91,10)
\psline[linecolor=lightgray,linewidth=1pt](91,10)(248,120)
\pspolygon[fillstyle=solid,fillcolor=LightBlue,linecolor=black](91,10)(28,249)(182,108)
\pscircle[fillstyle=solid,fillcolor=MidRed,linecolor=black](28,249){6}
\pscircle[fillstyle=solid,fillcolor=MidRed,linecolor=black](55,7){6}
\pscircle[fillstyle=solid,fillcolor=MidRed,linecolor=black](62,189){6}
\pscircle[fillstyle=solid,fillcolor=MidRed,linecolor=black](91,10){6}
\pscircle[fillstyle=solid,fillcolor=MidRed,linecolor=black](182,108){6}
\pscircle[fillstyle=solid,fillcolor=MidRed,linecolor=black](248,120){6}
}
\drawing{
\psline[linecolor=lightgray,linewidth=1pt](28,249)(248,120)
\psline[linecolor=lightgray,linewidth=1pt](28,249)(55,7)
\psline[linecolor=lightgray,linewidth=1pt](55,7)(91,10)
\psline[linecolor=lightgray,linewidth=1pt](91,10)(248,120)
\pspolygon[fillstyle=solid,fillcolor=LightBlue,linecolor=black](55,7)(28,249)(182,108)
\pspolygon[fillstyle=solid,fillcolor=LightBlue,linecolor=black](91,10)(55,7)(182,108)
\pscircle[fillstyle=solid,fillcolor=MidRed,linecolor=black](28,249){6}
\pscircle[fillstyle=solid,fillcolor=MidRed,linecolor=black](55,7){6}
\pscircle[fillstyle=solid,fillcolor=MidRed,linecolor=black](62,189){6}
\pscircle[fillstyle=solid,fillcolor=MidRed,linecolor=black](91,10){6}
\pscircle[fillstyle=solid,fillcolor=MidRed,linecolor=black](182,108){6}
\pscircle[fillstyle=solid,fillcolor=MidRed,linecolor=black](248,120){6}
}
\drawing{
\psline[linecolor=lightgray,linewidth=1pt](28,249)(248,120)
\psline[linecolor=lightgray,linewidth=1pt](28,249)(55,7)
\psline[linecolor=lightgray,linewidth=1pt](55,7)(91,10)
\psline[linecolor=lightgray,linewidth=1pt](91,10)(248,120)
\pspolygon[fillstyle=solid,fillcolor=LightBlue,linecolor=black](91,10)(28,249)(248,120)
\pscircle[fillstyle=solid,fillcolor=MidRed,linecolor=black](28,249){6}
\pscircle[fillstyle=solid,fillcolor=MidRed,linecolor=black](55,7){6}
\pscircle[fillstyle=solid,fillcolor=MidRed,linecolor=black](62,189){6}
\pscircle[fillstyle=solid,fillcolor=MidRed,linecolor=black](91,10){6}
\pscircle[fillstyle=solid,fillcolor=MidRed,linecolor=black](182,108){6}
\pscircle[fillstyle=solid,fillcolor=MidRed,linecolor=black](248,120){6}
}
\drawing{
\psline[linecolor=lightgray,linewidth=1pt](28,249)(248,120)
\psline[linecolor=lightgray,linewidth=1pt](28,249)(55,7)
\psline[linecolor=lightgray,linewidth=1pt](55,7)(91,10)
\psline[linecolor=lightgray,linewidth=1pt](91,10)(248,120)
\pspolygon[fillstyle=solid,fillcolor=LightBlue,linecolor=black](62,189)(182,108)(248,120)
\pspolygon[fillstyle=solid,fillcolor=LightBlue,linecolor=black](91,10)(62,189)(182,108)
\pspolygon[fillstyle=solid,fillcolor=LightBlue,linecolor=black](91,10)(28,249)(62,189)
\pspolygon[fillstyle=solid,fillcolor=LightBlue,linecolor=black](91,10)(182,108)(248,120)
\pspolygon[fillstyle=solid,fillcolor=LightBlue,linecolor=black](91,10)(55,7)(28,249)
\pscircle[fillstyle=solid,fillcolor=MidRed,linecolor=black](28,249){6}
\pscircle[fillstyle=solid,fillcolor=MidRed,linecolor=black](55,7){6}
\pscircle[fillstyle=solid,fillcolor=MidRed,linecolor=black](62,189){6}
\pscircle[fillstyle=solid,fillcolor=MidRed,linecolor=black](91,10){6}
\pscircle[fillstyle=solid,fillcolor=MidRed,linecolor=black](182,108){6}
\pscircle[fillstyle=solid,fillcolor=MidRed,linecolor=black](248,120){6}
}
\drawing{
\psline[linecolor=lightgray,linewidth=1pt](28,249)(248,120)
\psline[linecolor=lightgray,linewidth=1pt](28,249)(55,7)
\psline[linecolor=lightgray,linewidth=1pt](55,7)(91,10)
\psline[linecolor=lightgray,linewidth=1pt](91,10)(248,120)
\pspolygon[fillstyle=solid,fillcolor=LightBlue,linecolor=black](91,10)(62,189)(248,120)
\pspolygon[fillstyle=solid,fillcolor=LightBlue,linecolor=black](91,10)(55,7)(62,189)
\pscircle[fillstyle=solid,fillcolor=MidRed,linecolor=black](28,249){6}
\pscircle[fillstyle=solid,fillcolor=MidRed,linecolor=black](55,7){6}
\pscircle[fillstyle=solid,fillcolor=MidRed,linecolor=black](62,189){6}
\pscircle[fillstyle=solid,fillcolor=MidRed,linecolor=black](91,10){6}
\pscircle[fillstyle=solid,fillcolor=MidRed,linecolor=black](182,108){6}
\pscircle[fillstyle=solid,fillcolor=MidRed,linecolor=black](248,120){6}
}
\drawing{
\psline[linecolor=lightgray,linewidth=1pt](28,249)(248,120)
\psline[linecolor=lightgray,linewidth=1pt](28,249)(55,7)
\psline[linecolor=lightgray,linewidth=1pt](55,7)(91,10)
\psline[linecolor=lightgray,linewidth=1pt](91,10)(248,120)
\pspolygon[fillstyle=solid,fillcolor=LightBlue,linecolor=black](55,7)(28,249)(248,120)
\pscircle[fillstyle=solid,fillcolor=MidRed,linecolor=black](28,249){6}
\pscircle[fillstyle=solid,fillcolor=MidRed,linecolor=black](55,7){6}
\pscircle[fillstyle=solid,fillcolor=MidRed,linecolor=black](62,189){6}
\pscircle[fillstyle=solid,fillcolor=MidRed,linecolor=black](91,10){6}
\pscircle[fillstyle=solid,fillcolor=MidRed,linecolor=black](182,108){6}
\pscircle[fillstyle=solid,fillcolor=MidRed,linecolor=black](248,120){6}
}
\drawing{
\psline[linecolor=lightgray,linewidth=1pt](28,249)(248,120)
\psline[linecolor=lightgray,linewidth=1pt](28,249)(55,7)
\psline[linecolor=lightgray,linewidth=1pt](55,7)(91,10)
\psline[linecolor=lightgray,linewidth=1pt](91,10)(248,120)
\pspolygon[fillstyle=solid,fillcolor=LightBlue,linecolor=black](28,249)(182,108)(248,120)
\pspolygon[fillstyle=solid,fillcolor=LightBlue,linecolor=black](28,249)(62,189)(182,108)
\pspolygon[fillstyle=solid,fillcolor=LightBlue,linecolor=black](55,7)(62,189)(182,108)
\pspolygon[fillstyle=solid,fillcolor=LightBlue,linecolor=black](55,7)(182,108)(248,120)
\pspolygon[fillstyle=solid,fillcolor=LightBlue,linecolor=black](55,7)(28,249)(62,189)
\pspolygon[fillstyle=solid,fillcolor=LightBlue,linecolor=black](91,10)(55,7)(248,120)
\pscircle[fillstyle=solid,fillcolor=MidRed,linecolor=black](28,249){6}
\pscircle[fillstyle=solid,fillcolor=MidRed,linecolor=black](55,7){6}
\pscircle[fillstyle=solid,fillcolor=MidRed,linecolor=black](62,189){6}
\pscircle[fillstyle=solid,fillcolor=MidRed,linecolor=black](91,10){6}
\pscircle[fillstyle=solid,fillcolor=MidRed,linecolor=black](182,108){6}
\pscircle[fillstyle=solid,fillcolor=MidRed,linecolor=black](248,120){6}
}
\NewOrderType
\drawing{
\psline[linecolor=lightgray,linewidth=1pt](37,55)(73,201)
\psline[linecolor=lightgray,linewidth=1pt](73,201)(160,173)
\psline[linecolor=lightgray,linewidth=1pt](160,173)(219,83)
\psline[linecolor=lightgray,linewidth=1pt](37,55)(219,83)
\pspolygon[fillstyle=solid,fillcolor=LightBlue,linecolor=black](160,173)(119,94)(37,55)
\pspolygon[fillstyle=solid,fillcolor=LightBlue,linecolor=black](160,173)(81,135)(37,55)
\pscircle[fillstyle=solid,fillcolor=MidRed,linecolor=black](37,55){6}
\pscircle[fillstyle=solid,fillcolor=MidRed,linecolor=black](73,201){6}
\pscircle[fillstyle=solid,fillcolor=MidRed,linecolor=black](81,135){6}
\pscircle[fillstyle=solid,fillcolor=MidRed,linecolor=black](119,94){6}
\pscircle[fillstyle=solid,fillcolor=MidRed,linecolor=black](160,173){6}
\pscircle[fillstyle=solid,fillcolor=MidRed,linecolor=black](219,83){6}
}
\drawing{
\psline[linecolor=lightgray,linewidth=1pt](37,55)(73,201)
\psline[linecolor=lightgray,linewidth=1pt](73,201)(160,173)
\psline[linecolor=lightgray,linewidth=1pt](160,173)(219,83)
\psline[linecolor=lightgray,linewidth=1pt](37,55)(219,83)
\pspolygon[fillstyle=solid,fillcolor=LightBlue,linecolor=black](73,201)(119,94)(81,135)
\pspolygon[fillstyle=solid,fillcolor=LightBlue,linecolor=black](73,201)(219,83)(119,94)
\pscircle[fillstyle=solid,fillcolor=MidRed,linecolor=black](37,55){6}
\pscircle[fillstyle=solid,fillcolor=MidRed,linecolor=black](73,201){6}
\pscircle[fillstyle=solid,fillcolor=MidRed,linecolor=black](81,135){6}
\pscircle[fillstyle=solid,fillcolor=MidRed,linecolor=black](119,94){6}
\pscircle[fillstyle=solid,fillcolor=MidRed,linecolor=black](160,173){6}
\pscircle[fillstyle=solid,fillcolor=MidRed,linecolor=black](219,83){6}
}
\drawing{
\psline[linecolor=lightgray,linewidth=1pt](37,55)(73,201)
\psline[linecolor=lightgray,linewidth=1pt](73,201)(160,173)
\psline[linecolor=lightgray,linewidth=1pt](160,173)(219,83)
\psline[linecolor=lightgray,linewidth=1pt](37,55)(219,83)
\pspolygon[fillstyle=solid,fillcolor=LightBlue,linecolor=black](219,83)(119,94)(81,135)
\pspolygon[fillstyle=solid,fillcolor=LightBlue,linecolor=black](73,201)(219,83)(81,135)
\pscircle[fillstyle=solid,fillcolor=MidRed,linecolor=black](37,55){6}
\pscircle[fillstyle=solid,fillcolor=MidRed,linecolor=black](73,201){6}
\pscircle[fillstyle=solid,fillcolor=MidRed,linecolor=black](81,135){6}
\pscircle[fillstyle=solid,fillcolor=MidRed,linecolor=black](119,94){6}
\pscircle[fillstyle=solid,fillcolor=MidRed,linecolor=black](160,173){6}
\pscircle[fillstyle=solid,fillcolor=MidRed,linecolor=black](219,83){6}
}
\drawing{
\psline[linecolor=lightgray,linewidth=1pt](37,55)(73,201)
\psline[linecolor=lightgray,linewidth=1pt](73,201)(160,173)
\psline[linecolor=lightgray,linewidth=1pt](160,173)(219,83)
\psline[linecolor=lightgray,linewidth=1pt](37,55)(219,83)
\pspolygon[fillstyle=solid,fillcolor=LightBlue,linecolor=black](73,201)(119,94)(37,55)
\pspolygon[fillstyle=solid,fillcolor=LightBlue,linecolor=black](73,201)(160,173)(119,94)
\pscircle[fillstyle=solid,fillcolor=MidRed,linecolor=black](37,55){6}
\pscircle[fillstyle=solid,fillcolor=MidRed,linecolor=black](73,201){6}
\pscircle[fillstyle=solid,fillcolor=MidRed,linecolor=black](81,135){6}
\pscircle[fillstyle=solid,fillcolor=MidRed,linecolor=black](119,94){6}
\pscircle[fillstyle=solid,fillcolor=MidRed,linecolor=black](160,173){6}
\pscircle[fillstyle=solid,fillcolor=MidRed,linecolor=black](219,83){6}
}
\drawing{
\psline[linecolor=lightgray,linewidth=1pt](37,55)(73,201)
\psline[linecolor=lightgray,linewidth=1pt](73,201)(160,173)
\psline[linecolor=lightgray,linewidth=1pt](160,173)(219,83)
\psline[linecolor=lightgray,linewidth=1pt](37,55)(219,83)
\pspolygon[fillstyle=solid,fillcolor=LightBlue,linecolor=black](219,83)(81,135)(37,55)
\pspolygon[fillstyle=solid,fillcolor=LightBlue,linecolor=black](160,173)(219,83)(81,135)
\pscircle[fillstyle=solid,fillcolor=MidRed,linecolor=black](37,55){6}
\pscircle[fillstyle=solid,fillcolor=MidRed,linecolor=black](73,201){6}
\pscircle[fillstyle=solid,fillcolor=MidRed,linecolor=black](81,135){6}
\pscircle[fillstyle=solid,fillcolor=MidRed,linecolor=black](119,94){6}
\pscircle[fillstyle=solid,fillcolor=MidRed,linecolor=black](160,173){6}
\pscircle[fillstyle=solid,fillcolor=MidRed,linecolor=black](219,83){6}
}
\drawing{
\psline[linecolor=lightgray,linewidth=1pt](37,55)(73,201)
\psline[linecolor=lightgray,linewidth=1pt](73,201)(160,173)
\psline[linecolor=lightgray,linewidth=1pt](160,173)(219,83)
\psline[linecolor=lightgray,linewidth=1pt](37,55)(219,83)
\pspolygon[fillstyle=solid,fillcolor=LightBlue,linecolor=black](73,201)(219,83)(37,55)
\pspolygon[fillstyle=solid,fillcolor=LightBlue,linecolor=black](73,201)(160,173)(219,83)
\pscircle[fillstyle=solid,fillcolor=MidRed,linecolor=black](37,55){6}
\pscircle[fillstyle=solid,fillcolor=MidRed,linecolor=black](73,201){6}
\pscircle[fillstyle=solid,fillcolor=MidRed,linecolor=black](81,135){6}
\pscircle[fillstyle=solid,fillcolor=MidRed,linecolor=black](119,94){6}
\pscircle[fillstyle=solid,fillcolor=MidRed,linecolor=black](160,173){6}
\pscircle[fillstyle=solid,fillcolor=MidRed,linecolor=black](219,83){6}
}
\drawing{
\psline[linecolor=lightgray,linewidth=1pt](37,55)(73,201)
\psline[linecolor=lightgray,linewidth=1pt](73,201)(160,173)
\psline[linecolor=lightgray,linewidth=1pt](160,173)(219,83)
\psline[linecolor=lightgray,linewidth=1pt](37,55)(219,83)
\pspolygon[fillstyle=solid,fillcolor=LightBlue,linecolor=black](160,173)(219,83)(37,55)
\pspolygon[fillstyle=solid,fillcolor=LightBlue,linecolor=black](73,201)(160,173)(37,55)
\pscircle[fillstyle=solid,fillcolor=MidRed,linecolor=black](37,55){6}
\pscircle[fillstyle=solid,fillcolor=MidRed,linecolor=black](73,201){6}
\pscircle[fillstyle=solid,fillcolor=MidRed,linecolor=black](81,135){6}
\pscircle[fillstyle=solid,fillcolor=MidRed,linecolor=black](119,94){6}
\pscircle[fillstyle=solid,fillcolor=MidRed,linecolor=black](160,173){6}
\pscircle[fillstyle=solid,fillcolor=MidRed,linecolor=black](219,83){6}
}
\drawing{
\psline[linecolor=lightgray,linewidth=1pt](37,55)(73,201)
\psline[linecolor=lightgray,linewidth=1pt](73,201)(160,173)
\psline[linecolor=lightgray,linewidth=1pt](160,173)(219,83)
\psline[linecolor=lightgray,linewidth=1pt](37,55)(219,83)
\pspolygon[fillstyle=solid,fillcolor=LightBlue,linecolor=black](119,94)(81,135)(37,55)
\pspolygon[fillstyle=solid,fillcolor=LightBlue,linecolor=black](160,173)(119,94)(81,135)
\pspolygon[fillstyle=solid,fillcolor=LightBlue,linecolor=black](160,173)(219,83)(119,94)
\pspolygon[fillstyle=solid,fillcolor=LightBlue,linecolor=black](219,83)(119,94)(37,55)
\pspolygon[fillstyle=solid,fillcolor=LightBlue,linecolor=black](73,201)(160,173)(81,135)
\pspolygon[fillstyle=solid,fillcolor=LightBlue,linecolor=black](73,201)(81,135)(37,55)
\pscircle[fillstyle=solid,fillcolor=MidRed,linecolor=black](37,55){6}
\pscircle[fillstyle=solid,fillcolor=MidRed,linecolor=black](73,201){6}
\pscircle[fillstyle=solid,fillcolor=MidRed,linecolor=black](81,135){6}
\pscircle[fillstyle=solid,fillcolor=MidRed,linecolor=black](119,94){6}
\pscircle[fillstyle=solid,fillcolor=MidRed,linecolor=black](160,173){6}
\pscircle[fillstyle=solid,fillcolor=MidRed,linecolor=black](219,83){6}
}
\NewOrderType 
\drawing{
\psline[linecolor=lightgray,linewidth=1pt](15,85)(44,227)
\psline[linecolor=lightgray,linewidth=1pt](44,227)(241,99)
\psline[linecolor=lightgray,linewidth=1pt](15,85)(164,29)
\psline[linecolor=lightgray,linewidth=1pt](164,29)(241,99)
\pspolygon[fillstyle=solid,fillcolor=LightBlue,linecolor=black](15,85)(73,108)(44,227)
\pspolygon[fillstyle=solid,fillcolor=LightBlue,linecolor=black](164,29)(113,162)(44,227)
\pspolygon[fillstyle=solid,fillcolor=LightBlue,linecolor=black](164,29)(73,108)(44,227)
\pscircle[fillstyle=solid,fillcolor=MidRed,linecolor=black](15,85){6}
\pscircle[fillstyle=solid,fillcolor=MidRed,linecolor=black](44,227){6}
\pscircle[fillstyle=solid,fillcolor=MidRed,linecolor=black](73,108){6}
\pscircle[fillstyle=solid,fillcolor=MidRed,linecolor=black](113,162){6}
\pscircle[fillstyle=solid,fillcolor=MidRed,linecolor=black](164,29){6}
\pscircle[fillstyle=solid,fillcolor=MidRed,linecolor=black](241,99){6}
}
\drawing{
\psline[linecolor=lightgray,linewidth=1pt](15,85)(44,227)
\psline[linecolor=lightgray,linewidth=1pt](44,227)(241,99)
\psline[linecolor=lightgray,linewidth=1pt](15,85)(164,29)
\psline[linecolor=lightgray,linewidth=1pt](164,29)(241,99)
\pspolygon[fillstyle=solid,fillcolor=LightBlue,linecolor=black](241,99)(15,85)(113,162)
\pscircle[fillstyle=solid,fillcolor=MidRed,linecolor=black](15,85){6}
\pscircle[fillstyle=solid,fillcolor=MidRed,linecolor=black](44,227){6}
\pscircle[fillstyle=solid,fillcolor=MidRed,linecolor=black](73,108){6}
\pscircle[fillstyle=solid,fillcolor=MidRed,linecolor=black](113,162){6}
\pscircle[fillstyle=solid,fillcolor=MidRed,linecolor=black](164,29){6}
\pscircle[fillstyle=solid,fillcolor=MidRed,linecolor=black](241,99){6}
}
\drawing{
\psline[linecolor=lightgray,linewidth=1pt](15,85)(44,227)
\psline[linecolor=lightgray,linewidth=1pt](44,227)(241,99)
\psline[linecolor=lightgray,linewidth=1pt](15,85)(164,29)
\psline[linecolor=lightgray,linewidth=1pt](164,29)(241,99)
\pspolygon[fillstyle=solid,fillcolor=LightBlue,linecolor=black](164,29)(15,85)(113,162)
\pscircle[fillstyle=solid,fillcolor=MidRed,linecolor=black](15,85){6}
\pscircle[fillstyle=solid,fillcolor=MidRed,linecolor=black](44,227){6}
\pscircle[fillstyle=solid,fillcolor=MidRed,linecolor=black](73,108){6}
\pscircle[fillstyle=solid,fillcolor=MidRed,linecolor=black](113,162){6}
\pscircle[fillstyle=solid,fillcolor=MidRed,linecolor=black](164,29){6}
\pscircle[fillstyle=solid,fillcolor=MidRed,linecolor=black](241,99){6}
}
\drawing{
\psline[linecolor=lightgray,linewidth=1pt](15,85)(44,227)
\psline[linecolor=lightgray,linewidth=1pt](44,227)(241,99)
\psline[linecolor=lightgray,linewidth=1pt](15,85)(164,29)
\psline[linecolor=lightgray,linewidth=1pt](164,29)(241,99)
\pspolygon[fillstyle=solid,fillcolor=LightBlue,linecolor=black](241,99)(73,108)(44,227)
\pspolygon[fillstyle=solid,fillcolor=LightBlue,linecolor=black](241,99)(15,85)(73,108)
\pscircle[fillstyle=solid,fillcolor=MidRed,linecolor=black](15,85){6}
\pscircle[fillstyle=solid,fillcolor=MidRed,linecolor=black](44,227){6}
\pscircle[fillstyle=solid,fillcolor=MidRed,linecolor=black](73,108){6}
\pscircle[fillstyle=solid,fillcolor=MidRed,linecolor=black](113,162){6}
\pscircle[fillstyle=solid,fillcolor=MidRed,linecolor=black](164,29){6}
\pscircle[fillstyle=solid,fillcolor=MidRed,linecolor=black](241,99){6}
}
\drawing{
\psline[linecolor=lightgray,linewidth=1pt](15,85)(44,227)
\psline[linecolor=lightgray,linewidth=1pt](44,227)(241,99)
\psline[linecolor=lightgray,linewidth=1pt](15,85)(164,29)
\psline[linecolor=lightgray,linewidth=1pt](164,29)(241,99)
\pspolygon[fillstyle=solid,fillcolor=LightBlue,linecolor=black](241,99)(15,85)(44,227)
\pspolygon[fillstyle=solid,fillcolor=LightBlue,linecolor=black](241,99)(164,29)(15,85)
\pscircle[fillstyle=solid,fillcolor=MidRed,linecolor=black](15,85){6}
\pscircle[fillstyle=solid,fillcolor=MidRed,linecolor=black](44,227){6}
\pscircle[fillstyle=solid,fillcolor=MidRed,linecolor=black](73,108){6}
\pscircle[fillstyle=solid,fillcolor=MidRed,linecolor=black](113,162){6}
\pscircle[fillstyle=solid,fillcolor=MidRed,linecolor=black](164,29){6}
\pscircle[fillstyle=solid,fillcolor=MidRed,linecolor=black](241,99){6}
}
\drawing{
\psline[linecolor=lightgray,linewidth=1pt](15,85)(44,227)
\psline[linecolor=lightgray,linewidth=1pt](44,227)(241,99)
\psline[linecolor=lightgray,linewidth=1pt](15,85)(164,29)
\psline[linecolor=lightgray,linewidth=1pt](164,29)(241,99)
\pspolygon[fillstyle=solid,fillcolor=LightBlue,linecolor=black](73,108)(113,162)(44,227)
\pspolygon[fillstyle=solid,fillcolor=LightBlue,linecolor=black](241,99)(73,108)(113,162)
\pspolygon[fillstyle=solid,fillcolor=LightBlue,linecolor=black](241,99)(164,29)(73,108)
\pscircle[fillstyle=solid,fillcolor=MidRed,linecolor=black](15,85){6}
\pscircle[fillstyle=solid,fillcolor=MidRed,linecolor=black](44,227){6}
\pscircle[fillstyle=solid,fillcolor=MidRed,linecolor=black](73,108){6}
\pscircle[fillstyle=solid,fillcolor=MidRed,linecolor=black](113,162){6}
\pscircle[fillstyle=solid,fillcolor=MidRed,linecolor=black](164,29){6}
\pscircle[fillstyle=solid,fillcolor=MidRed,linecolor=black](241,99){6}
}
\drawing{
\psline[linecolor=lightgray,linewidth=1pt](15,85)(44,227)
\psline[linecolor=lightgray,linewidth=1pt](44,227)(241,99)
\psline[linecolor=lightgray,linewidth=1pt](15,85)(164,29)
\psline[linecolor=lightgray,linewidth=1pt](164,29)(241,99)
\pspolygon[fillstyle=solid,fillcolor=LightBlue,linecolor=black](164,29)(15,85)(44,227)
\pspolygon[fillstyle=solid,fillcolor=LightBlue,linecolor=black](241,99)(164,29)(44,227)
\pscircle[fillstyle=solid,fillcolor=MidRed,linecolor=black](15,85){6}
\pscircle[fillstyle=solid,fillcolor=MidRed,linecolor=black](44,227){6}
\pscircle[fillstyle=solid,fillcolor=MidRed,linecolor=black](73,108){6}
\pscircle[fillstyle=solid,fillcolor=MidRed,linecolor=black](113,162){6}
\pscircle[fillstyle=solid,fillcolor=MidRed,linecolor=black](164,29){6}
\pscircle[fillstyle=solid,fillcolor=MidRed,linecolor=black](241,99){6}
}
\drawing{
\psline[linecolor=lightgray,linewidth=1pt](15,85)(44,227)
\psline[linecolor=lightgray,linewidth=1pt](44,227)(241,99)
\psline[linecolor=lightgray,linewidth=1pt](15,85)(164,29)
\psline[linecolor=lightgray,linewidth=1pt](164,29)(241,99)
\pspolygon[fillstyle=solid,fillcolor=LightBlue,linecolor=black](15,85)(113,162)(44,227)
\pspolygon[fillstyle=solid,fillcolor=LightBlue,linecolor=black](15,85)(73,108)(113,162)
\pspolygon[fillstyle=solid,fillcolor=LightBlue,linecolor=black](164,29)(73,108)(113,162)
\pspolygon[fillstyle=solid,fillcolor=LightBlue,linecolor=black](164,29)(15,85)(73,108)
\pspolygon[fillstyle=solid,fillcolor=LightBlue,linecolor=black](241,99)(113,162)(44,227)
\pspolygon[fillstyle=solid,fillcolor=LightBlue,linecolor=black](241,99)(164,29)(113,162)
\pscircle[fillstyle=solid,fillcolor=MidRed,linecolor=black](15,85){6}
\pscircle[fillstyle=solid,fillcolor=MidRed,linecolor=black](44,227){6}
\pscircle[fillstyle=solid,fillcolor=MidRed,linecolor=black](73,108){6}
\pscircle[fillstyle=solid,fillcolor=MidRed,linecolor=black](113,162){6}
\pscircle[fillstyle=solid,fillcolor=MidRed,linecolor=black](164,29){6}
\pscircle[fillstyle=solid,fillcolor=MidRed,linecolor=black](241,99){6}
}
\NewOrderType 
\drawing{
\psline[linecolor=lightgray,linewidth=1pt](14,115)(105,238)
\psline[linecolor=lightgray,linewidth=1pt](105,238)(205,220)
\psline[linecolor=lightgray,linewidth=1pt](205,220)(242,19)
\psline[linecolor=lightgray,linewidth=1pt](14,115)(74,31)
\psline[linecolor=lightgray,linewidth=1pt](74,31)(242,19)
\pspolygon[fillstyle=solid,fillcolor=LightBlue,linecolor=black](105,238)(205,220)(118,142)
\pspolygon[fillstyle=solid,fillcolor=LightBlue,linecolor=black](14,115)(105,238)(118,142)
\pspolygon[fillstyle=solid,fillcolor=LightBlue,linecolor=black](14,115)(118,142)(242,19)
\pspolygon[fillstyle=solid,fillcolor=LightBlue,linecolor=black](205,220)(118,142)(242,19)
\pspolygon[fillstyle=solid,fillcolor=LightBlue,linecolor=black](14,115)(242,19)(74,31)
\pscircle[fillstyle=solid,fillcolor=MidRed,linecolor=black](14,115){6}
\pscircle[fillstyle=solid,fillcolor=MidRed,linecolor=black](74,31){6}
\pscircle[fillstyle=solid,fillcolor=MidRed,linecolor=black](105,238){6}
\pscircle[fillstyle=solid,fillcolor=MidRed,linecolor=black](118,142){6}
\pscircle[fillstyle=solid,fillcolor=MidRed,linecolor=black](205,220){6}
\pscircle[fillstyle=solid,fillcolor=MidRed,linecolor=black](242,19){6}
}
\drawing{
\psline[linecolor=lightgray,linewidth=1pt](14,115)(105,238)
\psline[linecolor=lightgray,linewidth=1pt](105,238)(205,220)
\psline[linecolor=lightgray,linewidth=1pt](205,220)(242,19)
\psline[linecolor=lightgray,linewidth=1pt](14,115)(74,31)
\psline[linecolor=lightgray,linewidth=1pt](74,31)(242,19)
\pspolygon[fillstyle=solid,fillcolor=LightBlue,linecolor=black](14,115)(205,220)(74,31)
\pscircle[fillstyle=solid,fillcolor=MidRed,linecolor=black](14,115){6}
\pscircle[fillstyle=solid,fillcolor=MidRed,linecolor=black](74,31){6}
\pscircle[fillstyle=solid,fillcolor=MidRed,linecolor=black](105,238){6}
\pscircle[fillstyle=solid,fillcolor=MidRed,linecolor=black](118,142){6}
\pscircle[fillstyle=solid,fillcolor=MidRed,linecolor=black](205,220){6}
\pscircle[fillstyle=solid,fillcolor=MidRed,linecolor=black](242,19){6}
}
\drawing{
\psline[linecolor=lightgray,linewidth=1pt](14,115)(105,238)
\psline[linecolor=lightgray,linewidth=1pt](105,238)(205,220)
\psline[linecolor=lightgray,linewidth=1pt](205,220)(242,19)
\psline[linecolor=lightgray,linewidth=1pt](14,115)(74,31)
\psline[linecolor=lightgray,linewidth=1pt](74,31)(242,19)
\pspolygon[fillstyle=solid,fillcolor=LightBlue,linecolor=black](14,115)(205,220)(242,19)
\pscircle[fillstyle=solid,fillcolor=MidRed,linecolor=black](14,115){6}
\pscircle[fillstyle=solid,fillcolor=MidRed,linecolor=black](74,31){6}
\pscircle[fillstyle=solid,fillcolor=MidRed,linecolor=black](105,238){6}
\pscircle[fillstyle=solid,fillcolor=MidRed,linecolor=black](118,142){6}
\pscircle[fillstyle=solid,fillcolor=MidRed,linecolor=black](205,220){6}
\pscircle[fillstyle=solid,fillcolor=MidRed,linecolor=black](242,19){6}
}
\drawing{
\psline[linecolor=lightgray,linewidth=1pt](14,115)(105,238)
\psline[linecolor=lightgray,linewidth=1pt](105,238)(205,220)
\psline[linecolor=lightgray,linewidth=1pt](205,220)(242,19)
\psline[linecolor=lightgray,linewidth=1pt](14,115)(74,31)
\psline[linecolor=lightgray,linewidth=1pt](74,31)(242,19)
\pspolygon[fillstyle=solid,fillcolor=LightBlue,linecolor=black](14,115)(105,238)(205,220)
\pspolygon[fillstyle=solid,fillcolor=LightBlue,linecolor=black](14,115)(205,220)(118,142)
\pspolygon[fillstyle=solid,fillcolor=LightBlue,linecolor=black](14,115)(118,142)(74,31)
\pspolygon[fillstyle=solid,fillcolor=LightBlue,linecolor=black](205,220)(118,142)(74,31)
\pscircle[fillstyle=solid,fillcolor=MidRed,linecolor=black](14,115){6}
\pscircle[fillstyle=solid,fillcolor=MidRed,linecolor=black](74,31){6}
\pscircle[fillstyle=solid,fillcolor=MidRed,linecolor=black](105,238){6}
\pscircle[fillstyle=solid,fillcolor=MidRed,linecolor=black](118,142){6}
\pscircle[fillstyle=solid,fillcolor=MidRed,linecolor=black](205,220){6}
\pscircle[fillstyle=solid,fillcolor=MidRed,linecolor=black](242,19){6}
}
\drawing{
\psline[linecolor=lightgray,linewidth=1pt](14,115)(105,238)
\psline[linecolor=lightgray,linewidth=1pt](105,238)(205,220)
\psline[linecolor=lightgray,linewidth=1pt](205,220)(242,19)
\psline[linecolor=lightgray,linewidth=1pt](14,115)(74,31)
\psline[linecolor=lightgray,linewidth=1pt](74,31)(242,19)
\pspolygon[fillstyle=solid,fillcolor=LightBlue,linecolor=black](105,238)(242,19)(74,31)
\pscircle[fillstyle=solid,fillcolor=MidRed,linecolor=black](14,115){6}
\pscircle[fillstyle=solid,fillcolor=MidRed,linecolor=black](74,31){6}
\pscircle[fillstyle=solid,fillcolor=MidRed,linecolor=black](105,238){6}
\pscircle[fillstyle=solid,fillcolor=MidRed,linecolor=black](118,142){6}
\pscircle[fillstyle=solid,fillcolor=MidRed,linecolor=black](205,220){6}
\pscircle[fillstyle=solid,fillcolor=MidRed,linecolor=black](242,19){6}
}
\drawing{
\psline[linecolor=lightgray,linewidth=1pt](14,115)(105,238)
\psline[linecolor=lightgray,linewidth=1pt](105,238)(205,220)
\psline[linecolor=lightgray,linewidth=1pt](205,220)(242,19)
\psline[linecolor=lightgray,linewidth=1pt](14,115)(74,31)
\psline[linecolor=lightgray,linewidth=1pt](74,31)(242,19)
\pspolygon[fillstyle=solid,fillcolor=LightBlue,linecolor=black](14,115)(105,238)(242,19)
\pscircle[fillstyle=solid,fillcolor=MidRed,linecolor=black](14,115){6}
\pscircle[fillstyle=solid,fillcolor=MidRed,linecolor=black](74,31){6}
\pscircle[fillstyle=solid,fillcolor=MidRed,linecolor=black](105,238){6}
\pscircle[fillstyle=solid,fillcolor=MidRed,linecolor=black](118,142){6}
\pscircle[fillstyle=solid,fillcolor=MidRed,linecolor=black](205,220){6}
\pscircle[fillstyle=solid,fillcolor=MidRed,linecolor=black](242,19){6}
}
\drawing{
\psline[linecolor=lightgray,linewidth=1pt](14,115)(105,238)
\psline[linecolor=lightgray,linewidth=1pt](105,238)(205,220)
\psline[linecolor=lightgray,linewidth=1pt](205,220)(242,19)
\psline[linecolor=lightgray,linewidth=1pt](14,115)(74,31)
\psline[linecolor=lightgray,linewidth=1pt](74,31)(242,19)
\pspolygon[fillstyle=solid,fillcolor=LightBlue,linecolor=black](105,238)(205,220)(74,31)
\pspolygon[fillstyle=solid,fillcolor=LightBlue,linecolor=black](205,220)(242,19)(74,31)
\pscircle[fillstyle=solid,fillcolor=MidRed,linecolor=black](14,115){6}
\pscircle[fillstyle=solid,fillcolor=MidRed,linecolor=black](74,31){6}
\pscircle[fillstyle=solid,fillcolor=MidRed,linecolor=black](105,238){6}
\pscircle[fillstyle=solid,fillcolor=MidRed,linecolor=black](118,142){6}
\pscircle[fillstyle=solid,fillcolor=MidRed,linecolor=black](205,220){6}
\pscircle[fillstyle=solid,fillcolor=MidRed,linecolor=black](242,19){6}
}
\drawing{
\psline[linecolor=lightgray,linewidth=1pt](14,115)(105,238)
\psline[linecolor=lightgray,linewidth=1pt](105,238)(205,220)
\psline[linecolor=lightgray,linewidth=1pt](205,220)(242,19)
\psline[linecolor=lightgray,linewidth=1pt](14,115)(74,31)
\psline[linecolor=lightgray,linewidth=1pt](74,31)(242,19)
\pspolygon[fillstyle=solid,fillcolor=LightBlue,linecolor=black](105,238)(118,142)(242,19)
\pspolygon[fillstyle=solid,fillcolor=LightBlue,linecolor=black](105,238)(118,142)(74,31)
\pspolygon[fillstyle=solid,fillcolor=LightBlue,linecolor=black](118,142)(242,19)(74,31)
\pspolygon[fillstyle=solid,fillcolor=LightBlue,linecolor=black](105,238)(205,220)(242,19)
\pspolygon[fillstyle=solid,fillcolor=LightBlue,linecolor=black](14,115)(105,238)(74,31)
\pscircle[fillstyle=solid,fillcolor=MidRed,linecolor=black](14,115){6}
\pscircle[fillstyle=solid,fillcolor=MidRed,linecolor=black](74,31){6}
\pscircle[fillstyle=solid,fillcolor=MidRed,linecolor=black](105,238){6}
\pscircle[fillstyle=solid,fillcolor=MidRed,linecolor=black](118,142){6}
\pscircle[fillstyle=solid,fillcolor=MidRed,linecolor=black](205,220){6}
\pscircle[fillstyle=solid,fillcolor=MidRed,linecolor=black](242,19){6}
}
\NewOrderType 
\drawing{
\psline[linecolor=lightgray,linewidth=1pt](2,227)(100,243)
\psline[linecolor=lightgray,linewidth=1pt](100,243)(235,209)
\psline[linecolor=lightgray,linewidth=1pt](2,227)(42,13)
\psline[linecolor=lightgray,linewidth=1pt](42,13)(235,209)
\pspolygon[fillstyle=solid,fillcolor=LightBlue,linecolor=black](100,243)(154,163)(42,13)
\pscircle[fillstyle=solid,fillcolor=MidRed,linecolor=black](2,227){6}
\pscircle[fillstyle=solid,fillcolor=MidRed,linecolor=black](42,13){6}
\pscircle[fillstyle=solid,fillcolor=MidRed,linecolor=black](100,243){6}
\pscircle[fillstyle=solid,fillcolor=MidRed,linecolor=black](105,202){6}
\pscircle[fillstyle=solid,fillcolor=MidRed,linecolor=black](154,163){6}
\pscircle[fillstyle=solid,fillcolor=MidRed,linecolor=black](235,209){6}
}
\drawing{
\psline[linecolor=lightgray,linewidth=1pt](2,227)(100,243)
\psline[linecolor=lightgray,linewidth=1pt](100,243)(235,209)
\psline[linecolor=lightgray,linewidth=1pt](2,227)(42,13)
\psline[linecolor=lightgray,linewidth=1pt](42,13)(235,209)
\pspolygon[fillstyle=solid,fillcolor=LightBlue,linecolor=black](2,227)(235,209)(42,13)
\pscircle[fillstyle=solid,fillcolor=MidRed,linecolor=black](2,227){6}
\pscircle[fillstyle=solid,fillcolor=MidRed,linecolor=black](42,13){6}
\pscircle[fillstyle=solid,fillcolor=MidRed,linecolor=black](100,243){6}
\pscircle[fillstyle=solid,fillcolor=MidRed,linecolor=black](105,202){6}
\pscircle[fillstyle=solid,fillcolor=MidRed,linecolor=black](154,163){6}
\pscircle[fillstyle=solid,fillcolor=MidRed,linecolor=black](235,209){6}
}
\drawing{
\psline[linecolor=lightgray,linewidth=1pt](2,227)(100,243)
\psline[linecolor=lightgray,linewidth=1pt](100,243)(235,209)
\psline[linecolor=lightgray,linewidth=1pt](2,227)(42,13)
\psline[linecolor=lightgray,linewidth=1pt](42,13)(235,209)
\pspolygon[fillstyle=solid,fillcolor=LightBlue,linecolor=black](2,227)(235,209)(154,163)
\pscircle[fillstyle=solid,fillcolor=MidRed,linecolor=black](2,227){6}
\pscircle[fillstyle=solid,fillcolor=MidRed,linecolor=black](42,13){6}
\pscircle[fillstyle=solid,fillcolor=MidRed,linecolor=black](100,243){6}
\pscircle[fillstyle=solid,fillcolor=MidRed,linecolor=black](105,202){6}
\pscircle[fillstyle=solid,fillcolor=MidRed,linecolor=black](154,163){6}
\pscircle[fillstyle=solid,fillcolor=MidRed,linecolor=black](235,209){6}
}
\drawing{
\psline[linecolor=lightgray,linewidth=1pt](2,227)(100,243)
\psline[linecolor=lightgray,linewidth=1pt](100,243)(235,209)
\psline[linecolor=lightgray,linewidth=1pt](2,227)(42,13)
\psline[linecolor=lightgray,linewidth=1pt](42,13)(235,209)
\pspolygon[fillstyle=solid,fillcolor=LightBlue,linecolor=black](2,227)(100,243)(154,163)
\pscircle[fillstyle=solid,fillcolor=MidRed,linecolor=black](2,227){6}
\pscircle[fillstyle=solid,fillcolor=MidRed,linecolor=black](42,13){6}
\pscircle[fillstyle=solid,fillcolor=MidRed,linecolor=black](100,243){6}
\pscircle[fillstyle=solid,fillcolor=MidRed,linecolor=black](105,202){6}
\pscircle[fillstyle=solid,fillcolor=MidRed,linecolor=black](154,163){6}
\pscircle[fillstyle=solid,fillcolor=MidRed,linecolor=black](235,209){6}
}
\drawing{
\psline[linecolor=lightgray,linewidth=1pt](2,227)(100,243)
\psline[linecolor=lightgray,linewidth=1pt](100,243)(235,209)
\psline[linecolor=lightgray,linewidth=1pt](2,227)(42,13)
\psline[linecolor=lightgray,linewidth=1pt](42,13)(235,209)
\pspolygon[fillstyle=solid,fillcolor=LightBlue,linecolor=black](100,243)(235,209)(42,13)
\pscircle[fillstyle=solid,fillcolor=MidRed,linecolor=black](2,227){6}
\pscircle[fillstyle=solid,fillcolor=MidRed,linecolor=black](42,13){6}
\pscircle[fillstyle=solid,fillcolor=MidRed,linecolor=black](100,243){6}
\pscircle[fillstyle=solid,fillcolor=MidRed,linecolor=black](105,202){6}
\pscircle[fillstyle=solid,fillcolor=MidRed,linecolor=black](154,163){6}
\pscircle[fillstyle=solid,fillcolor=MidRed,linecolor=black](235,209){6}
}
\drawing{
\psline[linecolor=lightgray,linewidth=1pt](2,227)(100,243)
\psline[linecolor=lightgray,linewidth=1pt](100,243)(235,209)
\psline[linecolor=lightgray,linewidth=1pt](2,227)(42,13)
\psline[linecolor=lightgray,linewidth=1pt](42,13)(235,209)
\pspolygon[fillstyle=solid,fillcolor=LightBlue,linecolor=black](2,227)(100,243)(235,209)
\pspolygon[fillstyle=solid,fillcolor=LightBlue,linecolor=black](2,227)(235,209)(105,202)
\pspolygon[fillstyle=solid,fillcolor=LightBlue,linecolor=black](2,227)(105,202)(42,13)
\pspolygon[fillstyle=solid,fillcolor=LightBlue,linecolor=black](235,209)(105,202)(42,13)
\pscircle[fillstyle=solid,fillcolor=MidRed,linecolor=black](2,227){6}
\pscircle[fillstyle=solid,fillcolor=MidRed,linecolor=black](42,13){6}
\pscircle[fillstyle=solid,fillcolor=MidRed,linecolor=black](100,243){6}
\pscircle[fillstyle=solid,fillcolor=MidRed,linecolor=black](105,202){6}
\pscircle[fillstyle=solid,fillcolor=MidRed,linecolor=black](154,163){6}
\pscircle[fillstyle=solid,fillcolor=MidRed,linecolor=black](235,209){6}
}
\drawing{
\psline[linecolor=lightgray,linewidth=1pt](2,227)(100,243)
\psline[linecolor=lightgray,linewidth=1pt](100,243)(235,209)
\psline[linecolor=lightgray,linewidth=1pt](2,227)(42,13)
\psline[linecolor=lightgray,linewidth=1pt](42,13)(235,209)
\pspolygon[fillstyle=solid,fillcolor=LightBlue,linecolor=black](100,243)(235,209)(105,202)
\pspolygon[fillstyle=solid,fillcolor=LightBlue,linecolor=black](2,227)(100,243)(105,202)
\pspolygon[fillstyle=solid,fillcolor=LightBlue,linecolor=black](2,227)(105,202)(154,163)
\pspolygon[fillstyle=solid,fillcolor=LightBlue,linecolor=black](235,209)(105,202)(154,163)
\pspolygon[fillstyle=solid,fillcolor=LightBlue,linecolor=black](2,227)(154,163)(42,13)
\pscircle[fillstyle=solid,fillcolor=MidRed,linecolor=black](2,227){6}
\pscircle[fillstyle=solid,fillcolor=MidRed,linecolor=black](42,13){6}
\pscircle[fillstyle=solid,fillcolor=MidRed,linecolor=black](100,243){6}
\pscircle[fillstyle=solid,fillcolor=MidRed,linecolor=black](105,202){6}
\pscircle[fillstyle=solid,fillcolor=MidRed,linecolor=black](154,163){6}
\pscircle[fillstyle=solid,fillcolor=MidRed,linecolor=black](235,209){6}
}
\drawing{
\psline[linecolor=lightgray,linewidth=1pt](2,227)(100,243)
\psline[linecolor=lightgray,linewidth=1pt](100,243)(235,209)
\psline[linecolor=lightgray,linewidth=1pt](2,227)(42,13)
\psline[linecolor=lightgray,linewidth=1pt](42,13)(235,209)
\pspolygon[fillstyle=solid,fillcolor=LightBlue,linecolor=black](100,243)(105,202)(154,163)
\pspolygon[fillstyle=solid,fillcolor=LightBlue,linecolor=black](100,243)(105,202)(42,13)
\pspolygon[fillstyle=solid,fillcolor=LightBlue,linecolor=black](105,202)(154,163)(42,13)
\pspolygon[fillstyle=solid,fillcolor=LightBlue,linecolor=black](100,243)(235,209)(154,163)
\pspolygon[fillstyle=solid,fillcolor=LightBlue,linecolor=black](235,209)(154,163)(42,13)
\pspolygon[fillstyle=solid,fillcolor=LightBlue,linecolor=black](2,227)(100,243)(42,13)
\pscircle[fillstyle=solid,fillcolor=MidRed,linecolor=black](2,227){6}
\pscircle[fillstyle=solid,fillcolor=MidRed,linecolor=black](42,13){6}
\pscircle[fillstyle=solid,fillcolor=MidRed,linecolor=black](100,243){6}
\pscircle[fillstyle=solid,fillcolor=MidRed,linecolor=black](105,202){6}
\pscircle[fillstyle=solid,fillcolor=MidRed,linecolor=black](154,163){6}
\pscircle[fillstyle=solid,fillcolor=MidRed,linecolor=black](235,209){6}
}
\NewOrderType 
\drawing{
\psline[linecolor=lightgray,linewidth=1pt](41,230)(207,211)
\psline[linecolor=lightgray,linewidth=1pt](207,211)(235,16)
\psline[linecolor=lightgray,linewidth=1pt](41,230)(130,25)
\psline[linecolor=lightgray,linewidth=1pt](130,25)(235,16)
\pspolygon[fillstyle=solid,fillcolor=LightBlue,linecolor=black](170,66)(130,25)(41,230)
\pspolygon[fillstyle=solid,fillcolor=LightBlue,linecolor=black](170,66)(89,197)(41,230)
\pspolygon[fillstyle=solid,fillcolor=LightBlue,linecolor=black](235,16)(170,66)(89,197)
\pspolygon[fillstyle=solid,fillcolor=LightBlue,linecolor=black](207,211)(235,16)(89,197)
\pscircle[fillstyle=solid,fillcolor=MidRed,linecolor=black](41,230){6}
\pscircle[fillstyle=solid,fillcolor=MidRed,linecolor=black](89,197){6}
\pscircle[fillstyle=solid,fillcolor=MidRed,linecolor=black](130,25){6}
\pscircle[fillstyle=solid,fillcolor=MidRed,linecolor=black](170,66){6}
\pscircle[fillstyle=solid,fillcolor=MidRed,linecolor=black](207,211){6}
\pscircle[fillstyle=solid,fillcolor=MidRed,linecolor=black](235,16){6}
}
\drawing{
\psline[linecolor=lightgray,linewidth=1pt](41,230)(207,211)
\psline[linecolor=lightgray,linewidth=1pt](207,211)(235,16)
\psline[linecolor=lightgray,linewidth=1pt](41,230)(130,25)
\psline[linecolor=lightgray,linewidth=1pt](130,25)(235,16)
\pspolygon[fillstyle=solid,fillcolor=LightBlue,linecolor=black](235,16)(130,25)(89,197)
\pscircle[fillstyle=solid,fillcolor=MidRed,linecolor=black](41,230){6}
\pscircle[fillstyle=solid,fillcolor=MidRed,linecolor=black](89,197){6}
\pscircle[fillstyle=solid,fillcolor=MidRed,linecolor=black](130,25){6}
\pscircle[fillstyle=solid,fillcolor=MidRed,linecolor=black](170,66){6}
\pscircle[fillstyle=solid,fillcolor=MidRed,linecolor=black](207,211){6}
\pscircle[fillstyle=solid,fillcolor=MidRed,linecolor=black](235,16){6}
}
\drawing{
\psline[linecolor=lightgray,linewidth=1pt](41,230)(207,211)
\psline[linecolor=lightgray,linewidth=1pt](207,211)(235,16)
\psline[linecolor=lightgray,linewidth=1pt](41,230)(130,25)
\psline[linecolor=lightgray,linewidth=1pt](130,25)(235,16)
\pspolygon[fillstyle=solid,fillcolor=LightBlue,linecolor=black](207,211)(130,25)(89,197)
\pspolygon[fillstyle=solid,fillcolor=LightBlue,linecolor=black](207,211)(235,16)(130,25)
\pscircle[fillstyle=solid,fillcolor=MidRed,linecolor=black](41,230){6}
\pscircle[fillstyle=solid,fillcolor=MidRed,linecolor=black](89,197){6}
\pscircle[fillstyle=solid,fillcolor=MidRed,linecolor=black](130,25){6}
\pscircle[fillstyle=solid,fillcolor=MidRed,linecolor=black](170,66){6}
\pscircle[fillstyle=solid,fillcolor=MidRed,linecolor=black](207,211){6}
\pscircle[fillstyle=solid,fillcolor=MidRed,linecolor=black](235,16){6}
}
\drawing{
\psline[linecolor=lightgray,linewidth=1pt](41,230)(207,211)
\psline[linecolor=lightgray,linewidth=1pt](207,211)(235,16)
\psline[linecolor=lightgray,linewidth=1pt](41,230)(130,25)
\psline[linecolor=lightgray,linewidth=1pt](130,25)(235,16)
\pspolygon[fillstyle=solid,fillcolor=LightBlue,linecolor=black](207,211)(235,16)(41,230)
\pspolygon[fillstyle=solid,fillcolor=LightBlue,linecolor=black](235,16)(130,25)(41,230)
\pscircle[fillstyle=solid,fillcolor=MidRed,linecolor=black](41,230){6}
\pscircle[fillstyle=solid,fillcolor=MidRed,linecolor=black](89,197){6}
\pscircle[fillstyle=solid,fillcolor=MidRed,linecolor=black](130,25){6}
\pscircle[fillstyle=solid,fillcolor=MidRed,linecolor=black](170,66){6}
\pscircle[fillstyle=solid,fillcolor=MidRed,linecolor=black](207,211){6}
\pscircle[fillstyle=solid,fillcolor=MidRed,linecolor=black](235,16){6}
}
\drawing{
\psline[linecolor=lightgray,linewidth=1pt](41,230)(207,211)
\psline[linecolor=lightgray,linewidth=1pt](207,211)(235,16)
\psline[linecolor=lightgray,linewidth=1pt](41,230)(130,25)
\psline[linecolor=lightgray,linewidth=1pt](130,25)(235,16)
\pspolygon[fillstyle=solid,fillcolor=LightBlue,linecolor=black](207,211)(170,66)(41,230)
\pscircle[fillstyle=solid,fillcolor=MidRed,linecolor=black](41,230){6}
\pscircle[fillstyle=solid,fillcolor=MidRed,linecolor=black](89,197){6}
\pscircle[fillstyle=solid,fillcolor=MidRed,linecolor=black](130,25){6}
\pscircle[fillstyle=solid,fillcolor=MidRed,linecolor=black](170,66){6}
\pscircle[fillstyle=solid,fillcolor=MidRed,linecolor=black](207,211){6}
\pscircle[fillstyle=solid,fillcolor=MidRed,linecolor=black](235,16){6}
}
\drawing{
\psline[linecolor=lightgray,linewidth=1pt](41,230)(207,211)
\psline[linecolor=lightgray,linewidth=1pt](207,211)(235,16)
\psline[linecolor=lightgray,linewidth=1pt](41,230)(130,25)
\psline[linecolor=lightgray,linewidth=1pt](130,25)(235,16)
\pspolygon[fillstyle=solid,fillcolor=LightBlue,linecolor=black](235,16)(170,66)(41,230)
\pspolygon[fillstyle=solid,fillcolor=LightBlue,linecolor=black](235,16)(89,197)(41,230)
\pscircle[fillstyle=solid,fillcolor=MidRed,linecolor=black](41,230){6}
\pscircle[fillstyle=solid,fillcolor=MidRed,linecolor=black](89,197){6}
\pscircle[fillstyle=solid,fillcolor=MidRed,linecolor=black](130,25){6}
\pscircle[fillstyle=solid,fillcolor=MidRed,linecolor=black](170,66){6}
\pscircle[fillstyle=solid,fillcolor=MidRed,linecolor=black](207,211){6}
\pscircle[fillstyle=solid,fillcolor=MidRed,linecolor=black](235,16){6}
}
\drawing{
\psline[linecolor=lightgray,linewidth=1pt](41,230)(207,211)
\psline[linecolor=lightgray,linewidth=1pt](207,211)(235,16)
\psline[linecolor=lightgray,linewidth=1pt](41,230)(130,25)
\psline[linecolor=lightgray,linewidth=1pt](130,25)(235,16)
\pspolygon[fillstyle=solid,fillcolor=LightBlue,linecolor=black](207,211)(130,25)(41,230)
\pspolygon[fillstyle=solid,fillcolor=LightBlue,linecolor=black](207,211)(170,66)(130,25)
\pscircle[fillstyle=solid,fillcolor=MidRed,linecolor=black](41,230){6}
\pscircle[fillstyle=solid,fillcolor=MidRed,linecolor=black](89,197){6}
\pscircle[fillstyle=solid,fillcolor=MidRed,linecolor=black](130,25){6}
\pscircle[fillstyle=solid,fillcolor=MidRed,linecolor=black](170,66){6}
\pscircle[fillstyle=solid,fillcolor=MidRed,linecolor=black](207,211){6}
\pscircle[fillstyle=solid,fillcolor=MidRed,linecolor=black](235,16){6}
}
\drawing{
\psline[linecolor=lightgray,linewidth=1pt](41,230)(207,211)
\psline[linecolor=lightgray,linewidth=1pt](207,211)(235,16)
\psline[linecolor=lightgray,linewidth=1pt](41,230)(130,25)
\psline[linecolor=lightgray,linewidth=1pt](130,25)(235,16)
\pspolygon[fillstyle=solid,fillcolor=LightBlue,linecolor=black](130,25)(89,197)(41,230)
\pspolygon[fillstyle=solid,fillcolor=LightBlue,linecolor=black](170,66)(130,25)(89,197)
\pspolygon[fillstyle=solid,fillcolor=LightBlue,linecolor=black](207,211)(170,66)(89,197)
\pspolygon[fillstyle=solid,fillcolor=LightBlue,linecolor=black](207,211)(89,197)(41,230)
\pspolygon[fillstyle=solid,fillcolor=LightBlue,linecolor=black](207,211)(235,16)(170,66)
\pspolygon[fillstyle=solid,fillcolor=LightBlue,linecolor=black](235,16)(170,66)(130,25)
\pscircle[fillstyle=solid,fillcolor=MidRed,linecolor=black](41,230){6}
\pscircle[fillstyle=solid,fillcolor=MidRed,linecolor=black](89,197){6}
\pscircle[fillstyle=solid,fillcolor=MidRed,linecolor=black](130,25){6}
\pscircle[fillstyle=solid,fillcolor=MidRed,linecolor=black](170,66){6}
\pscircle[fillstyle=solid,fillcolor=MidRed,linecolor=black](207,211){6}
\pscircle[fillstyle=solid,fillcolor=MidRed,linecolor=black](235,16){6}
}
\NewOrderType 
\drawing{
\psline[linecolor=lightgray,linewidth=1pt](2,238)(245,218)
\psline[linecolor=lightgray,linewidth=1pt](2,238)(236,27)
\psline[linecolor=lightgray,linewidth=1pt](236,27)(245,218)
\pspolygon[fillstyle=solid,fillcolor=LightBlue,linecolor=black](2,238)(198,120)(236,27)
\pspolygon[fillstyle=solid,fillcolor=LightBlue,linecolor=black](2,238)(245,218)(198,120)
\pspolygon[fillstyle=solid,fillcolor=LightBlue,linecolor=black](245,218)(198,120)(236,27)
\pscircle[fillstyle=solid,fillcolor=MidRed,linecolor=black](2,238){6}
\pscircle[fillstyle=solid,fillcolor=MidRed,linecolor=black](125,147){6}
\pscircle[fillstyle=solid,fillcolor=MidRed,linecolor=black](198,120){6}
\pscircle[fillstyle=solid,fillcolor=MidRed,linecolor=black](227,130){6}
\pscircle[fillstyle=solid,fillcolor=MidRed,linecolor=black](236,27){6}
\pscircle[fillstyle=solid,fillcolor=MidRed,linecolor=black](245,218){6}
}
\drawing{
\psline[linecolor=lightgray,linewidth=1pt](2,238)(245,218)
\psline[linecolor=lightgray,linewidth=1pt](2,238)(236,27)
\psline[linecolor=lightgray,linewidth=1pt](236,27)(245,218)
\pspolygon[fillstyle=solid,fillcolor=LightBlue,linecolor=black](245,218)(198,120)(125,147)
\pspolygon[fillstyle=solid,fillcolor=LightBlue,linecolor=black](245,218)(227,130)(198,120)
\pscircle[fillstyle=solid,fillcolor=MidRed,linecolor=black](2,238){6}
\pscircle[fillstyle=solid,fillcolor=MidRed,linecolor=black](125,147){6}
\pscircle[fillstyle=solid,fillcolor=MidRed,linecolor=black](198,120){6}
\pscircle[fillstyle=solid,fillcolor=MidRed,linecolor=black](227,130){6}
\pscircle[fillstyle=solid,fillcolor=MidRed,linecolor=black](236,27){6}
\pscircle[fillstyle=solid,fillcolor=MidRed,linecolor=black](245,218){6}
}
\drawing{
\psline[linecolor=lightgray,linewidth=1pt](2,238)(245,218)
\psline[linecolor=lightgray,linewidth=1pt](2,238)(236,27)
\psline[linecolor=lightgray,linewidth=1pt](236,27)(245,218)
\pspolygon[fillstyle=solid,fillcolor=LightBlue,linecolor=black](227,130)(198,120)(125,147)
\pspolygon[fillstyle=solid,fillcolor=LightBlue,linecolor=black](245,218)(227,130)(125,147)
\pscircle[fillstyle=solid,fillcolor=MidRed,linecolor=black](2,238){6}
\pscircle[fillstyle=solid,fillcolor=MidRed,linecolor=black](125,147){6}
\pscircle[fillstyle=solid,fillcolor=MidRed,linecolor=black](198,120){6}
\pscircle[fillstyle=solid,fillcolor=MidRed,linecolor=black](227,130){6}
\pscircle[fillstyle=solid,fillcolor=MidRed,linecolor=black](236,27){6}
\pscircle[fillstyle=solid,fillcolor=MidRed,linecolor=black](245,218){6}
}
\drawing{
\psline[linecolor=lightgray,linewidth=1pt](2,238)(245,218)
\psline[linecolor=lightgray,linewidth=1pt](2,238)(236,27)
\psline[linecolor=lightgray,linewidth=1pt](236,27)(245,218)
\pspolygon[fillstyle=solid,fillcolor=LightBlue,linecolor=black](227,130)(125,147)(236,27)
\pspolygon[fillstyle=solid,fillcolor=LightBlue,linecolor=black](2,238)(227,130)(125,147)
\pscircle[fillstyle=solid,fillcolor=MidRed,linecolor=black](2,238){6}
\pscircle[fillstyle=solid,fillcolor=MidRed,linecolor=black](125,147){6}
\pscircle[fillstyle=solid,fillcolor=MidRed,linecolor=black](198,120){6}
\pscircle[fillstyle=solid,fillcolor=MidRed,linecolor=black](227,130){6}
\pscircle[fillstyle=solid,fillcolor=MidRed,linecolor=black](236,27){6}
\pscircle[fillstyle=solid,fillcolor=MidRed,linecolor=black](245,218){6}
}
\drawing{
\psline[linecolor=lightgray,linewidth=1pt](2,238)(245,218)
\psline[linecolor=lightgray,linewidth=1pt](2,238)(236,27)
\psline[linecolor=lightgray,linewidth=1pt](236,27)(245,218)
\pspolygon[fillstyle=solid,fillcolor=LightBlue,linecolor=black](2,238)(245,218)(236,27)
\pscircle[fillstyle=solid,fillcolor=MidRed,linecolor=black](2,238){6}
\pscircle[fillstyle=solid,fillcolor=MidRed,linecolor=black](125,147){6}
\pscircle[fillstyle=solid,fillcolor=MidRed,linecolor=black](198,120){6}
\pscircle[fillstyle=solid,fillcolor=MidRed,linecolor=black](227,130){6}
\pscircle[fillstyle=solid,fillcolor=MidRed,linecolor=black](236,27){6}
\pscircle[fillstyle=solid,fillcolor=MidRed,linecolor=black](245,218){6}
}
\drawing{
\psline[linecolor=lightgray,linewidth=1pt](2,238)(245,218)
\psline[linecolor=lightgray,linewidth=1pt](2,238)(236,27)
\psline[linecolor=lightgray,linewidth=1pt](236,27)(245,218)
\pspolygon[fillstyle=solid,fillcolor=LightBlue,linecolor=black](245,218)(125,147)(236,27)
\pspolygon[fillstyle=solid,fillcolor=LightBlue,linecolor=black](2,238)(245,218)(125,147)
\pscircle[fillstyle=solid,fillcolor=MidRed,linecolor=black](2,238){6}
\pscircle[fillstyle=solid,fillcolor=MidRed,linecolor=black](125,147){6}
\pscircle[fillstyle=solid,fillcolor=MidRed,linecolor=black](198,120){6}
\pscircle[fillstyle=solid,fillcolor=MidRed,linecolor=black](227,130){6}
\pscircle[fillstyle=solid,fillcolor=MidRed,linecolor=black](236,27){6}
\pscircle[fillstyle=solid,fillcolor=MidRed,linecolor=black](245,218){6}
}
\drawing{
\psline[linecolor=lightgray,linewidth=1pt](2,238)(245,218)
\psline[linecolor=lightgray,linewidth=1pt](2,238)(236,27)
\psline[linecolor=lightgray,linewidth=1pt](236,27)(245,218)
\pspolygon[fillstyle=solid,fillcolor=LightBlue,linecolor=black](2,238)(227,130)(236,27)
\pscircle[fillstyle=solid,fillcolor=MidRed,linecolor=black](2,238){6}
\pscircle[fillstyle=solid,fillcolor=MidRed,linecolor=black](125,147){6}
\pscircle[fillstyle=solid,fillcolor=MidRed,linecolor=black](198,120){6}
\pscircle[fillstyle=solid,fillcolor=MidRed,linecolor=black](227,130){6}
\pscircle[fillstyle=solid,fillcolor=MidRed,linecolor=black](236,27){6}
\pscircle[fillstyle=solid,fillcolor=MidRed,linecolor=black](245,218){6}
}
\drawing{
\psline[linecolor=lightgray,linewidth=1pt](2,238)(245,218)
\psline[linecolor=lightgray,linewidth=1pt](2,238)(236,27)
\psline[linecolor=lightgray,linewidth=1pt](236,27)(245,218)
\pspolygon[fillstyle=solid,fillcolor=LightBlue,linecolor=black](198,120)(125,147)(236,27)
\pspolygon[fillstyle=solid,fillcolor=LightBlue,linecolor=black](2,238)(125,147)(236,27)
\pspolygon[fillstyle=solid,fillcolor=LightBlue,linecolor=black](2,238)(198,120)(125,147)
\pspolygon[fillstyle=solid,fillcolor=LightBlue,linecolor=black](2,238)(227,130)(198,120)
\pspolygon[fillstyle=solid,fillcolor=LightBlue,linecolor=black](227,130)(198,120)(236,27)
\pspolygon[fillstyle=solid,fillcolor=LightBlue,linecolor=black](2,238)(245,218)(227,130)
\pspolygon[fillstyle=solid,fillcolor=LightBlue,linecolor=black](245,218)(227,130)(236,27)
\pscircle[fillstyle=solid,fillcolor=MidRed,linecolor=black](2,238){6}
\pscircle[fillstyle=solid,fillcolor=MidRed,linecolor=black](125,147){6}
\pscircle[fillstyle=solid,fillcolor=MidRed,linecolor=black](198,120){6}
\pscircle[fillstyle=solid,fillcolor=MidRed,linecolor=black](227,130){6}
\pscircle[fillstyle=solid,fillcolor=MidRed,linecolor=black](236,27){6}
\pscircle[fillstyle=solid,fillcolor=MidRed,linecolor=black](245,218){6}
}
\NewOrderType 
\drawing{
\psline[linecolor=lightgray,linewidth=1pt](3,72)(221,251)
\psline[linecolor=lightgray,linewidth=1pt](221,251)(244,5)
\psline[linecolor=lightgray,linewidth=1pt](3,72)(244,5)
\pspolygon[fillstyle=solid,fillcolor=LightBlue,linecolor=black](144,167)(178,137)(221,251)
\pspolygon[fillstyle=solid,fillcolor=LightBlue,linecolor=black](178,137)(217,134)(221,251)
\pspolygon[fillstyle=solid,fillcolor=LightBlue,linecolor=black](3,72)(144,167)(178,137)
\pspolygon[fillstyle=solid,fillcolor=LightBlue,linecolor=black](3,72)(178,137)(217,134)
\pspolygon[fillstyle=solid,fillcolor=LightBlue,linecolor=black](244,5)(3,72)(217,134)
\pscircle[fillstyle=solid,fillcolor=MidRed,linecolor=black](3,72){6}
\pscircle[fillstyle=solid,fillcolor=MidRed,linecolor=black](144,167){6}
\pscircle[fillstyle=solid,fillcolor=MidRed,linecolor=black](178,137){6}
\pscircle[fillstyle=solid,fillcolor=MidRed,linecolor=black](217,134){6}
\pscircle[fillstyle=solid,fillcolor=MidRed,linecolor=black](221,251){6}
\pscircle[fillstyle=solid,fillcolor=MidRed,linecolor=black](244,5){6}
}
\drawing{
\psline[linecolor=lightgray,linewidth=1pt](3,72)(221,251)
\psline[linecolor=lightgray,linewidth=1pt](221,251)(244,5)
\psline[linecolor=lightgray,linewidth=1pt](3,72)(244,5)
\pspolygon[fillstyle=solid,fillcolor=LightBlue,linecolor=black](3,72)(217,134)(221,251)
\pscircle[fillstyle=solid,fillcolor=MidRed,linecolor=black](3,72){6}
\pscircle[fillstyle=solid,fillcolor=MidRed,linecolor=black](144,167){6}
\pscircle[fillstyle=solid,fillcolor=MidRed,linecolor=black](178,137){6}
\pscircle[fillstyle=solid,fillcolor=MidRed,linecolor=black](217,134){6}
\pscircle[fillstyle=solid,fillcolor=MidRed,linecolor=black](221,251){6}
\pscircle[fillstyle=solid,fillcolor=MidRed,linecolor=black](244,5){6}
}
\drawing{
\psline[linecolor=lightgray,linewidth=1pt](3,72)(221,251)
\psline[linecolor=lightgray,linewidth=1pt](221,251)(244,5)
\psline[linecolor=lightgray,linewidth=1pt](3,72)(244,5)
\pspolygon[fillstyle=solid,fillcolor=LightBlue,linecolor=black](3,72)(144,167)(217,134)
\pscircle[fillstyle=solid,fillcolor=MidRed,linecolor=black](3,72){6}
\pscircle[fillstyle=solid,fillcolor=MidRed,linecolor=black](144,167){6}
\pscircle[fillstyle=solid,fillcolor=MidRed,linecolor=black](178,137){6}
\pscircle[fillstyle=solid,fillcolor=MidRed,linecolor=black](217,134){6}
\pscircle[fillstyle=solid,fillcolor=MidRed,linecolor=black](221,251){6}
\pscircle[fillstyle=solid,fillcolor=MidRed,linecolor=black](244,5){6}
}
\drawing{
\psline[linecolor=lightgray,linewidth=1pt](3,72)(221,251)
\psline[linecolor=lightgray,linewidth=1pt](221,251)(244,5)
\psline[linecolor=lightgray,linewidth=1pt](3,72)(244,5)
\pspolygon[fillstyle=solid,fillcolor=LightBlue,linecolor=black](244,5)(178,137)(221,251)
\pspolygon[fillstyle=solid,fillcolor=LightBlue,linecolor=black](3,72)(178,137)(221,251)
\pspolygon[fillstyle=solid,fillcolor=LightBlue,linecolor=black](244,5)(3,72)(178,137)
\pscircle[fillstyle=solid,fillcolor=MidRed,linecolor=black](3,72){6}
\pscircle[fillstyle=solid,fillcolor=MidRed,linecolor=black](144,167){6}
\pscircle[fillstyle=solid,fillcolor=MidRed,linecolor=black](178,137){6}
\pscircle[fillstyle=solid,fillcolor=MidRed,linecolor=black](217,134){6}
\pscircle[fillstyle=solid,fillcolor=MidRed,linecolor=black](221,251){6}
\pscircle[fillstyle=solid,fillcolor=MidRed,linecolor=black](244,5){6}
}
\drawing{
\psline[linecolor=lightgray,linewidth=1pt](3,72)(221,251)
\psline[linecolor=lightgray,linewidth=1pt](221,251)(244,5)
\psline[linecolor=lightgray,linewidth=1pt](3,72)(244,5)
\pspolygon[fillstyle=solid,fillcolor=LightBlue,linecolor=black](244,5)(144,167)(221,251)
\pscircle[fillstyle=solid,fillcolor=MidRed,linecolor=black](3,72){6}
\pscircle[fillstyle=solid,fillcolor=MidRed,linecolor=black](144,167){6}
\pscircle[fillstyle=solid,fillcolor=MidRed,linecolor=black](178,137){6}
\pscircle[fillstyle=solid,fillcolor=MidRed,linecolor=black](217,134){6}
\pscircle[fillstyle=solid,fillcolor=MidRed,linecolor=black](221,251){6}
\pscircle[fillstyle=solid,fillcolor=MidRed,linecolor=black](244,5){6}
}
\drawing{
\psline[linecolor=lightgray,linewidth=1pt](3,72)(221,251)
\psline[linecolor=lightgray,linewidth=1pt](221,251)(244,5)
\psline[linecolor=lightgray,linewidth=1pt](3,72)(244,5)
\pspolygon[fillstyle=solid,fillcolor=LightBlue,linecolor=black](244,5)(3,72)(221,251)
\pscircle[fillstyle=solid,fillcolor=MidRed,linecolor=black](3,72){6}
\pscircle[fillstyle=solid,fillcolor=MidRed,linecolor=black](144,167){6}
\pscircle[fillstyle=solid,fillcolor=MidRed,linecolor=black](178,137){6}
\pscircle[fillstyle=solid,fillcolor=MidRed,linecolor=black](217,134){6}
\pscircle[fillstyle=solid,fillcolor=MidRed,linecolor=black](221,251){6}
\pscircle[fillstyle=solid,fillcolor=MidRed,linecolor=black](244,5){6}
}
\drawing{
\psline[linecolor=lightgray,linewidth=1pt](3,72)(221,251)
\psline[linecolor=lightgray,linewidth=1pt](221,251)(244,5)
\psline[linecolor=lightgray,linewidth=1pt](3,72)(244,5)
\pspolygon[fillstyle=solid,fillcolor=LightBlue,linecolor=black](244,5)(144,167)(217,134)
\pscircle[fillstyle=solid,fillcolor=MidRed,linecolor=black](3,72){6}
\pscircle[fillstyle=solid,fillcolor=MidRed,linecolor=black](144,167){6}
\pscircle[fillstyle=solid,fillcolor=MidRed,linecolor=black](178,137){6}
\pscircle[fillstyle=solid,fillcolor=MidRed,linecolor=black](217,134){6}
\pscircle[fillstyle=solid,fillcolor=MidRed,linecolor=black](221,251){6}
\pscircle[fillstyle=solid,fillcolor=MidRed,linecolor=black](244,5){6}
}
\drawing{
\psline[linecolor=lightgray,linewidth=1pt](3,72)(221,251)
\psline[linecolor=lightgray,linewidth=1pt](221,251)(244,5)
\psline[linecolor=lightgray,linewidth=1pt](3,72)(244,5)
\pspolygon[fillstyle=solid,fillcolor=LightBlue,linecolor=black](144,167)(178,137)(217,134)
\pspolygon[fillstyle=solid,fillcolor=LightBlue,linecolor=black](244,5)(144,167)(178,137)
\pspolygon[fillstyle=solid,fillcolor=LightBlue,linecolor=black](244,5)(178,137)(217,134)
\pspolygon[fillstyle=solid,fillcolor=LightBlue,linecolor=black](144,167)(217,134)(221,251)
\pspolygon[fillstyle=solid,fillcolor=LightBlue,linecolor=black](244,5)(217,134)(221,251)
\pspolygon[fillstyle=solid,fillcolor=LightBlue,linecolor=black](244,5)(3,72)(144,167)
\pspolygon[fillstyle=solid,fillcolor=LightBlue,linecolor=black](3,72)(144,167)(221,251)
\pscircle[fillstyle=solid,fillcolor=MidRed,linecolor=black](3,72){6}
\pscircle[fillstyle=solid,fillcolor=MidRed,linecolor=black](144,167){6}
\pscircle[fillstyle=solid,fillcolor=MidRed,linecolor=black](178,137){6}
\pscircle[fillstyle=solid,fillcolor=MidRed,linecolor=black](217,134){6}
\pscircle[fillstyle=solid,fillcolor=MidRed,linecolor=black](221,251){6}
\pscircle[fillstyle=solid,fillcolor=MidRed,linecolor=black](244,5){6}
}
\NewOrderType 
\drawing{
\psline[linecolor=lightgray,linewidth=1pt](15,7)(94,249)
\psline[linecolor=lightgray,linewidth=1pt](94,249)(240,32)
\psline[linecolor=lightgray,linewidth=1pt](15,7)(240,32)
\pspolygon[fillstyle=solid,fillcolor=LightBlue,linecolor=black](15,7)(180,75)(94,249)
\pscircle[fillstyle=solid,fillcolor=MidRed,linecolor=black](15,7){6}
\pscircle[fillstyle=solid,fillcolor=MidRed,linecolor=black](61,102){6}
\pscircle[fillstyle=solid,fillcolor=MidRed,linecolor=black](94,249){6}
\pscircle[fillstyle=solid,fillcolor=MidRed,linecolor=black](118,164){6}
\pscircle[fillstyle=solid,fillcolor=MidRed,linecolor=black](180,75){6}
\pscircle[fillstyle=solid,fillcolor=MidRed,linecolor=black](240,32){6}
}
\drawing{
\psline[linecolor=lightgray,linewidth=1pt](15,7)(94,249)
\psline[linecolor=lightgray,linewidth=1pt](94,249)(240,32)
\psline[linecolor=lightgray,linewidth=1pt](15,7)(240,32)
\pspolygon[fillstyle=solid,fillcolor=LightBlue,linecolor=black](240,32)(61,102)(94,249)
\pscircle[fillstyle=solid,fillcolor=MidRed,linecolor=black](15,7){6}
\pscircle[fillstyle=solid,fillcolor=MidRed,linecolor=black](61,102){6}
\pscircle[fillstyle=solid,fillcolor=MidRed,linecolor=black](94,249){6}
\pscircle[fillstyle=solid,fillcolor=MidRed,linecolor=black](118,164){6}
\pscircle[fillstyle=solid,fillcolor=MidRed,linecolor=black](180,75){6}
\pscircle[fillstyle=solid,fillcolor=MidRed,linecolor=black](240,32){6}
}
\drawing{
\psline[linecolor=lightgray,linewidth=1pt](15,7)(94,249)
\psline[linecolor=lightgray,linewidth=1pt](94,249)(240,32)
\psline[linecolor=lightgray,linewidth=1pt](15,7)(240,32)
\pspolygon[fillstyle=solid,fillcolor=LightBlue,linecolor=black](240,32)(61,102)(180,75)
\pspolygon[fillstyle=solid,fillcolor=LightBlue,linecolor=black](61,102)(180,75)(94,249)
\pscircle[fillstyle=solid,fillcolor=MidRed,linecolor=black](15,7){6}
\pscircle[fillstyle=solid,fillcolor=MidRed,linecolor=black](61,102){6}
\pscircle[fillstyle=solid,fillcolor=MidRed,linecolor=black](94,249){6}
\pscircle[fillstyle=solid,fillcolor=MidRed,linecolor=black](118,164){6}
\pscircle[fillstyle=solid,fillcolor=MidRed,linecolor=black](180,75){6}
\pscircle[fillstyle=solid,fillcolor=MidRed,linecolor=black](240,32){6}
}
\drawing{
\psline[linecolor=lightgray,linewidth=1pt](15,7)(94,249)
\psline[linecolor=lightgray,linewidth=1pt](94,249)(240,32)
\psline[linecolor=lightgray,linewidth=1pt](15,7)(240,32)
\pspolygon[fillstyle=solid,fillcolor=LightBlue,linecolor=black](240,32)(15,7)(61,102)
\pspolygon[fillstyle=solid,fillcolor=LightBlue,linecolor=black](240,32)(61,102)(118,164)
\pscircle[fillstyle=solid,fillcolor=MidRed,linecolor=black](15,7){6}
\pscircle[fillstyle=solid,fillcolor=MidRed,linecolor=black](61,102){6}
\pscircle[fillstyle=solid,fillcolor=MidRed,linecolor=black](94,249){6}
\pscircle[fillstyle=solid,fillcolor=MidRed,linecolor=black](118,164){6}
\pscircle[fillstyle=solid,fillcolor=MidRed,linecolor=black](180,75){6}
\pscircle[fillstyle=solid,fillcolor=MidRed,linecolor=black](240,32){6}
}
\drawing{
\psline[linecolor=lightgray,linewidth=1pt](15,7)(94,249)
\psline[linecolor=lightgray,linewidth=1pt](94,249)(240,32)
\psline[linecolor=lightgray,linewidth=1pt](15,7)(240,32)
\pspolygon[fillstyle=solid,fillcolor=LightBlue,linecolor=black](15,7)(61,102)(180,75)
\pspolygon[fillstyle=solid,fillcolor=LightBlue,linecolor=black](61,102)(180,75)(118,164)
\pspolygon[fillstyle=solid,fillcolor=LightBlue,linecolor=black](240,32)(180,75)(118,164)
\pscircle[fillstyle=solid,fillcolor=MidRed,linecolor=black](15,7){6}
\pscircle[fillstyle=solid,fillcolor=MidRed,linecolor=black](61,102){6}
\pscircle[fillstyle=solid,fillcolor=MidRed,linecolor=black](94,249){6}
\pscircle[fillstyle=solid,fillcolor=MidRed,linecolor=black](118,164){6}
\pscircle[fillstyle=solid,fillcolor=MidRed,linecolor=black](180,75){6}
\pscircle[fillstyle=solid,fillcolor=MidRed,linecolor=black](240,32){6}
}
\drawing{
\psline[linecolor=lightgray,linewidth=1pt](15,7)(94,249)
\psline[linecolor=lightgray,linewidth=1pt](94,249)(240,32)
\psline[linecolor=lightgray,linewidth=1pt](15,7)(240,32)
\pspolygon[fillstyle=solid,fillcolor=LightBlue,linecolor=black](240,32)(15,7)(94,249)
\pscircle[fillstyle=solid,fillcolor=MidRed,linecolor=black](15,7){6}
\pscircle[fillstyle=solid,fillcolor=MidRed,linecolor=black](61,102){6}
\pscircle[fillstyle=solid,fillcolor=MidRed,linecolor=black](94,249){6}
\pscircle[fillstyle=solid,fillcolor=MidRed,linecolor=black](118,164){6}
\pscircle[fillstyle=solid,fillcolor=MidRed,linecolor=black](180,75){6}
\pscircle[fillstyle=solid,fillcolor=MidRed,linecolor=black](240,32){6}
}
\drawing{
\psline[linecolor=lightgray,linewidth=1pt](15,7)(94,249)
\psline[linecolor=lightgray,linewidth=1pt](94,249)(240,32)
\psline[linecolor=lightgray,linewidth=1pt](15,7)(240,32)
\pspolygon[fillstyle=solid,fillcolor=LightBlue,linecolor=black](15,7)(61,102)(118,164)
\pspolygon[fillstyle=solid,fillcolor=LightBlue,linecolor=black](15,7)(61,102)(94,249)
\pspolygon[fillstyle=solid,fillcolor=LightBlue,linecolor=black](240,32)(15,7)(118,164)
\pspolygon[fillstyle=solid,fillcolor=LightBlue,linecolor=black](61,102)(118,164)(94,249)
\pspolygon[fillstyle=solid,fillcolor=LightBlue,linecolor=black](240,32)(118,164)(94,249)
\pscircle[fillstyle=solid,fillcolor=MidRed,linecolor=black](15,7){6}
\pscircle[fillstyle=solid,fillcolor=MidRed,linecolor=black](61,102){6}
\pscircle[fillstyle=solid,fillcolor=MidRed,linecolor=black](94,249){6}
\pscircle[fillstyle=solid,fillcolor=MidRed,linecolor=black](118,164){6}
\pscircle[fillstyle=solid,fillcolor=MidRed,linecolor=black](180,75){6}
\pscircle[fillstyle=solid,fillcolor=MidRed,linecolor=black](240,32){6}
}
\drawing{
\psline[linecolor=lightgray,linewidth=1pt](15,7)(94,249)
\psline[linecolor=lightgray,linewidth=1pt](94,249)(240,32)
\psline[linecolor=lightgray,linewidth=1pt](15,7)(240,32)
\pspolygon[fillstyle=solid,fillcolor=LightBlue,linecolor=black](15,7)(118,164)(94,249)
\pspolygon[fillstyle=solid,fillcolor=LightBlue,linecolor=black](15,7)(180,75)(118,164)
\pspolygon[fillstyle=solid,fillcolor=LightBlue,linecolor=black](180,75)(118,164)(94,249)
\pspolygon[fillstyle=solid,fillcolor=LightBlue,linecolor=black](240,32)(180,75)(94,249)
\pspolygon[fillstyle=solid,fillcolor=LightBlue,linecolor=black](240,32)(15,7)(180,75)
\pscircle[fillstyle=solid,fillcolor=MidRed,linecolor=black](15,7){6}
\pscircle[fillstyle=solid,fillcolor=MidRed,linecolor=black](61,102){6}
\pscircle[fillstyle=solid,fillcolor=MidRed,linecolor=black](94,249){6}
\pscircle[fillstyle=solid,fillcolor=MidRed,linecolor=black](118,164){6}
\pscircle[fillstyle=solid,fillcolor=MidRed,linecolor=black](180,75){6}
\pscircle[fillstyle=solid,fillcolor=MidRed,linecolor=black](240,32){6}
}
\NewOrderType 
\drawing{
\psline[linecolor=lightgray,linewidth=1pt](20,9)(21,247)
\psline[linecolor=lightgray,linewidth=1pt](21,247)(235,103)
\psline[linecolor=lightgray,linewidth=1pt](20,9)(235,103)
\pspolygon[fillstyle=solid,fillcolor=LightBlue,linecolor=black](133,81)(86,122)(20,9)
\pspolygon[fillstyle=solid,fillcolor=LightBlue,linecolor=black](21,247)(86,122)(20,9)
\pspolygon[fillstyle=solid,fillcolor=LightBlue,linecolor=black](235,103)(133,81)(86,122)
\pspolygon[fillstyle=solid,fillcolor=LightBlue,linecolor=black](21,247)(235,103)(86,122)
\pscircle[fillstyle=solid,fillcolor=MidRed,linecolor=black](20,9){6}
\pscircle[fillstyle=solid,fillcolor=MidRed,linecolor=black](21,247){6}
\pscircle[fillstyle=solid,fillcolor=MidRed,linecolor=black](50,93){6}
\pscircle[fillstyle=solid,fillcolor=MidRed,linecolor=black](86,122){6}
\pscircle[fillstyle=solid,fillcolor=MidRed,linecolor=black](133,81){6}
\pscircle[fillstyle=solid,fillcolor=MidRed,linecolor=black](235,103){6}
}
\drawing{
\psline[linecolor=lightgray,linewidth=1pt](20,9)(21,247)
\psline[linecolor=lightgray,linewidth=1pt](21,247)(235,103)
\psline[linecolor=lightgray,linewidth=1pt](20,9)(235,103)
\pspolygon[fillstyle=solid,fillcolor=LightBlue,linecolor=black](235,103)(133,81)(50,93)
\pspolygon[fillstyle=solid,fillcolor=LightBlue,linecolor=black](235,103)(86,122)(50,93)
\pscircle[fillstyle=solid,fillcolor=MidRed,linecolor=black](20,9){6}
\pscircle[fillstyle=solid,fillcolor=MidRed,linecolor=black](21,247){6}
\pscircle[fillstyle=solid,fillcolor=MidRed,linecolor=black](50,93){6}
\pscircle[fillstyle=solid,fillcolor=MidRed,linecolor=black](86,122){6}
\pscircle[fillstyle=solid,fillcolor=MidRed,linecolor=black](133,81){6}
\pscircle[fillstyle=solid,fillcolor=MidRed,linecolor=black](235,103){6}
}
\drawing{
\psline[linecolor=lightgray,linewidth=1pt](20,9)(21,247)
\psline[linecolor=lightgray,linewidth=1pt](21,247)(235,103)
\psline[linecolor=lightgray,linewidth=1pt](20,9)(235,103)
\pspolygon[fillstyle=solid,fillcolor=LightBlue,linecolor=black](21,247)(133,81)(50,93)
\pscircle[fillstyle=solid,fillcolor=MidRed,linecolor=black](20,9){6}
\pscircle[fillstyle=solid,fillcolor=MidRed,linecolor=black](21,247){6}
\pscircle[fillstyle=solid,fillcolor=MidRed,linecolor=black](50,93){6}
\pscircle[fillstyle=solid,fillcolor=MidRed,linecolor=black](86,122){6}
\pscircle[fillstyle=solid,fillcolor=MidRed,linecolor=black](133,81){6}
\pscircle[fillstyle=solid,fillcolor=MidRed,linecolor=black](235,103){6}
}
\drawing{
\psline[linecolor=lightgray,linewidth=1pt](20,9)(21,247)
\psline[linecolor=lightgray,linewidth=1pt](21,247)(235,103)
\psline[linecolor=lightgray,linewidth=1pt](20,9)(235,103)
\pspolygon[fillstyle=solid,fillcolor=LightBlue,linecolor=black](235,103)(50,93)(20,9)
\pspolygon[fillstyle=solid,fillcolor=LightBlue,linecolor=black](21,247)(235,103)(50,93)
\pscircle[fillstyle=solid,fillcolor=MidRed,linecolor=black](20,9){6}
\pscircle[fillstyle=solid,fillcolor=MidRed,linecolor=black](21,247){6}
\pscircle[fillstyle=solid,fillcolor=MidRed,linecolor=black](50,93){6}
\pscircle[fillstyle=solid,fillcolor=MidRed,linecolor=black](86,122){6}
\pscircle[fillstyle=solid,fillcolor=MidRed,linecolor=black](133,81){6}
\pscircle[fillstyle=solid,fillcolor=MidRed,linecolor=black](235,103){6}
}
\drawing{
\psline[linecolor=lightgray,linewidth=1pt](20,9)(21,247)
\psline[linecolor=lightgray,linewidth=1pt](21,247)(235,103)
\psline[linecolor=lightgray,linewidth=1pt](20,9)(235,103)
\pspolygon[fillstyle=solid,fillcolor=LightBlue,linecolor=black](21,247)(235,103)(20,9)
\pscircle[fillstyle=solid,fillcolor=MidRed,linecolor=black](20,9){6}
\pscircle[fillstyle=solid,fillcolor=MidRed,linecolor=black](21,247){6}
\pscircle[fillstyle=solid,fillcolor=MidRed,linecolor=black](50,93){6}
\pscircle[fillstyle=solid,fillcolor=MidRed,linecolor=black](86,122){6}
\pscircle[fillstyle=solid,fillcolor=MidRed,linecolor=black](133,81){6}
\pscircle[fillstyle=solid,fillcolor=MidRed,linecolor=black](235,103){6}
}
\drawing{
\psline[linecolor=lightgray,linewidth=1pt](20,9)(21,247)
\psline[linecolor=lightgray,linewidth=1pt](21,247)(235,103)
\psline[linecolor=lightgray,linewidth=1pt](20,9)(235,103)
\pspolygon[fillstyle=solid,fillcolor=LightBlue,linecolor=black](235,103)(86,122)(20,9)
\pspolygon[fillstyle=solid,fillcolor=LightBlue,linecolor=black](86,122)(50,93)(20,9)
\pscircle[fillstyle=solid,fillcolor=MidRed,linecolor=black](20,9){6}
\pscircle[fillstyle=solid,fillcolor=MidRed,linecolor=black](21,247){6}
\pscircle[fillstyle=solid,fillcolor=MidRed,linecolor=black](50,93){6}
\pscircle[fillstyle=solid,fillcolor=MidRed,linecolor=black](86,122){6}
\pscircle[fillstyle=solid,fillcolor=MidRed,linecolor=black](133,81){6}
\pscircle[fillstyle=solid,fillcolor=MidRed,linecolor=black](235,103){6}
}
\drawing{
\psline[linecolor=lightgray,linewidth=1pt](20,9)(21,247)
\psline[linecolor=lightgray,linewidth=1pt](21,247)(235,103)
\psline[linecolor=lightgray,linewidth=1pt](20,9)(235,103)
\pspolygon[fillstyle=solid,fillcolor=LightBlue,linecolor=black](21,247)(133,81)(20,9)
\pscircle[fillstyle=solid,fillcolor=MidRed,linecolor=black](20,9){6}
\pscircle[fillstyle=solid,fillcolor=MidRed,linecolor=black](21,247){6}
\pscircle[fillstyle=solid,fillcolor=MidRed,linecolor=black](50,93){6}
\pscircle[fillstyle=solid,fillcolor=MidRed,linecolor=black](86,122){6}
\pscircle[fillstyle=solid,fillcolor=MidRed,linecolor=black](133,81){6}
\pscircle[fillstyle=solid,fillcolor=MidRed,linecolor=black](235,103){6}
}
\drawing{
\psline[linecolor=lightgray,linewidth=1pt](20,9)(21,247)
\psline[linecolor=lightgray,linewidth=1pt](21,247)(235,103)
\psline[linecolor=lightgray,linewidth=1pt](20,9)(235,103)
\pspolygon[fillstyle=solid,fillcolor=LightBlue,linecolor=black](133,81)(50,93)(20,9)
\pspolygon[fillstyle=solid,fillcolor=LightBlue,linecolor=black](133,81)(86,122)(50,93)
\pspolygon[fillstyle=solid,fillcolor=LightBlue,linecolor=black](21,247)(133,81)(86,122)
\pspolygon[fillstyle=solid,fillcolor=LightBlue,linecolor=black](21,247)(50,93)(20,9)
\pspolygon[fillstyle=solid,fillcolor=LightBlue,linecolor=black](21,247)(86,122)(50,93)
\pspolygon[fillstyle=solid,fillcolor=LightBlue,linecolor=black](21,247)(235,103)(133,81)
\pspolygon[fillstyle=solid,fillcolor=LightBlue,linecolor=black](235,103)(133,81)(20,9)
\pscircle[fillstyle=solid,fillcolor=MidRed,linecolor=black](20,9){6}
\pscircle[fillstyle=solid,fillcolor=MidRed,linecolor=black](21,247){6}
\pscircle[fillstyle=solid,fillcolor=MidRed,linecolor=black](50,93){6}
\pscircle[fillstyle=solid,fillcolor=MidRed,linecolor=black](86,122){6}
\pscircle[fillstyle=solid,fillcolor=MidRed,linecolor=black](133,81){6}
\pscircle[fillstyle=solid,fillcolor=MidRed,linecolor=black](235,103){6}
}
\NewOrderType 
\drawing{
\psline[linecolor=lightgray,linewidth=1pt](5,240)(253,136)
\psline[linecolor=lightgray,linewidth=1pt](5,240)(65,15)
\psline[linecolor=lightgray,linewidth=1pt](65,15)(253,136)
\pspolygon[fillstyle=solid,fillcolor=LightBlue,linecolor=black](194,131)(63,182)(65,15)
\pspolygon[fillstyle=solid,fillcolor=LightBlue,linecolor=black](5,240)(194,131)(63,182)
\pspolygon[fillstyle=solid,fillcolor=LightBlue,linecolor=black](5,240)(253,136)(194,131)
\pspolygon[fillstyle=solid,fillcolor=LightBlue,linecolor=black](5,240)(63,182)(65,15)
\pscircle[fillstyle=solid,fillcolor=MidRed,linecolor=black](5,240){6}
\pscircle[fillstyle=solid,fillcolor=MidRed,linecolor=black](63,182){6}
\pscircle[fillstyle=solid,fillcolor=MidRed,linecolor=black](65,15){6}
\pscircle[fillstyle=solid,fillcolor=MidRed,linecolor=black](101,83){6}
\pscircle[fillstyle=solid,fillcolor=MidRed,linecolor=black](194,131){6}
\pscircle[fillstyle=solid,fillcolor=MidRed,linecolor=black](253,136){6}
}
\drawing{
\psline[linecolor=lightgray,linewidth=1pt](5,240)(253,136)
\psline[linecolor=lightgray,linewidth=1pt](5,240)(65,15)
\psline[linecolor=lightgray,linewidth=1pt](65,15)(253,136)
\pspolygon[fillstyle=solid,fillcolor=LightBlue,linecolor=black](5,240)(253,136)(65,15)
\pscircle[fillstyle=solid,fillcolor=MidRed,linecolor=black](5,240){6}
\pscircle[fillstyle=solid,fillcolor=MidRed,linecolor=black](63,182){6}
\pscircle[fillstyle=solid,fillcolor=MidRed,linecolor=black](65,15){6}
\pscircle[fillstyle=solid,fillcolor=MidRed,linecolor=black](101,83){6}
\pscircle[fillstyle=solid,fillcolor=MidRed,linecolor=black](194,131){6}
\pscircle[fillstyle=solid,fillcolor=MidRed,linecolor=black](253,136){6}
}
\drawing{
\psline[linecolor=lightgray,linewidth=1pt](5,240)(253,136)
\psline[linecolor=lightgray,linewidth=1pt](5,240)(65,15)
\psline[linecolor=lightgray,linewidth=1pt](65,15)(253,136)
\pspolygon[fillstyle=solid,fillcolor=LightBlue,linecolor=black](5,240)(253,136)(101,83)
\pscircle[fillstyle=solid,fillcolor=MidRed,linecolor=black](5,240){6}
\pscircle[fillstyle=solid,fillcolor=MidRed,linecolor=black](63,182){6}
\pscircle[fillstyle=solid,fillcolor=MidRed,linecolor=black](65,15){6}
\pscircle[fillstyle=solid,fillcolor=MidRed,linecolor=black](101,83){6}
\pscircle[fillstyle=solid,fillcolor=MidRed,linecolor=black](194,131){6}
\pscircle[fillstyle=solid,fillcolor=MidRed,linecolor=black](253,136){6}
}
\drawing{
\psline[linecolor=lightgray,linewidth=1pt](5,240)(253,136)
\psline[linecolor=lightgray,linewidth=1pt](5,240)(65,15)
\psline[linecolor=lightgray,linewidth=1pt](65,15)(253,136)
\pspolygon[fillstyle=solid,fillcolor=LightBlue,linecolor=black](253,136)(63,182)(65,15)
\pscircle[fillstyle=solid,fillcolor=MidRed,linecolor=black](5,240){6}
\pscircle[fillstyle=solid,fillcolor=MidRed,linecolor=black](63,182){6}
\pscircle[fillstyle=solid,fillcolor=MidRed,linecolor=black](65,15){6}
\pscircle[fillstyle=solid,fillcolor=MidRed,linecolor=black](101,83){6}
\pscircle[fillstyle=solid,fillcolor=MidRed,linecolor=black](194,131){6}
\pscircle[fillstyle=solid,fillcolor=MidRed,linecolor=black](253,136){6}
}
\drawing{
\psline[linecolor=lightgray,linewidth=1pt](5,240)(253,136)
\psline[linecolor=lightgray,linewidth=1pt](5,240)(65,15)
\psline[linecolor=lightgray,linewidth=1pt](65,15)(253,136)
\pspolygon[fillstyle=solid,fillcolor=LightBlue,linecolor=black](253,136)(63,182)(101,83)
\pspolygon[fillstyle=solid,fillcolor=LightBlue,linecolor=black](63,182)(101,83)(65,15)
\pscircle[fillstyle=solid,fillcolor=MidRed,linecolor=black](5,240){6}
\pscircle[fillstyle=solid,fillcolor=MidRed,linecolor=black](63,182){6}
\pscircle[fillstyle=solid,fillcolor=MidRed,linecolor=black](65,15){6}
\pscircle[fillstyle=solid,fillcolor=MidRed,linecolor=black](101,83){6}
\pscircle[fillstyle=solid,fillcolor=MidRed,linecolor=black](194,131){6}
\pscircle[fillstyle=solid,fillcolor=MidRed,linecolor=black](253,136){6}
}
\drawing{
\psline[linecolor=lightgray,linewidth=1pt](5,240)(253,136)
\psline[linecolor=lightgray,linewidth=1pt](5,240)(65,15)
\psline[linecolor=lightgray,linewidth=1pt](65,15)(253,136)
\pspolygon[fillstyle=solid,fillcolor=LightBlue,linecolor=black](5,240)(194,131)(65,15)
\pscircle[fillstyle=solid,fillcolor=MidRed,linecolor=black](5,240){6}
\pscircle[fillstyle=solid,fillcolor=MidRed,linecolor=black](63,182){6}
\pscircle[fillstyle=solid,fillcolor=MidRed,linecolor=black](65,15){6}
\pscircle[fillstyle=solid,fillcolor=MidRed,linecolor=black](101,83){6}
\pscircle[fillstyle=solid,fillcolor=MidRed,linecolor=black](194,131){6}
\pscircle[fillstyle=solid,fillcolor=MidRed,linecolor=black](253,136){6}
}
\drawing{
\psline[linecolor=lightgray,linewidth=1pt](5,240)(253,136)
\psline[linecolor=lightgray,linewidth=1pt](5,240)(65,15)
\psline[linecolor=lightgray,linewidth=1pt](65,15)(253,136)
\pspolygon[fillstyle=solid,fillcolor=LightBlue,linecolor=black](253,136)(101,83)(65,15)
\pspolygon[fillstyle=solid,fillcolor=LightBlue,linecolor=black](253,136)(194,131)(101,83)
\pspolygon[fillstyle=solid,fillcolor=LightBlue,linecolor=black](5,240)(194,131)(101,83)
\pscircle[fillstyle=solid,fillcolor=MidRed,linecolor=black](5,240){6}
\pscircle[fillstyle=solid,fillcolor=MidRed,linecolor=black](63,182){6}
\pscircle[fillstyle=solid,fillcolor=MidRed,linecolor=black](65,15){6}
\pscircle[fillstyle=solid,fillcolor=MidRed,linecolor=black](101,83){6}
\pscircle[fillstyle=solid,fillcolor=MidRed,linecolor=black](194,131){6}
\pscircle[fillstyle=solid,fillcolor=MidRed,linecolor=black](253,136){6}
}
\drawing{
\psline[linecolor=lightgray,linewidth=1pt](5,240)(253,136)
\psline[linecolor=lightgray,linewidth=1pt](5,240)(65,15)
\psline[linecolor=lightgray,linewidth=1pt](65,15)(253,136)
\pspolygon[fillstyle=solid,fillcolor=LightBlue,linecolor=black](194,131)(101,83)(65,15)
\pspolygon[fillstyle=solid,fillcolor=LightBlue,linecolor=black](194,131)(63,182)(101,83)
\pspolygon[fillstyle=solid,fillcolor=LightBlue,linecolor=black](253,136)(194,131)(63,182)
\pspolygon[fillstyle=solid,fillcolor=LightBlue,linecolor=black](5,240)(63,182)(101,83)
\pspolygon[fillstyle=solid,fillcolor=LightBlue,linecolor=black](5,240)(101,83)(65,15)
\pspolygon[fillstyle=solid,fillcolor=LightBlue,linecolor=black](253,136)(194,131)(65,15)
\pspolygon[fillstyle=solid,fillcolor=LightBlue,linecolor=black](5,240)(253,136)(63,182)
\pscircle[fillstyle=solid,fillcolor=MidRed,linecolor=black](5,240){6}
\pscircle[fillstyle=solid,fillcolor=MidRed,linecolor=black](63,182){6}
\pscircle[fillstyle=solid,fillcolor=MidRed,linecolor=black](65,15){6}
\pscircle[fillstyle=solid,fillcolor=MidRed,linecolor=black](101,83){6}
\pscircle[fillstyle=solid,fillcolor=MidRed,linecolor=black](194,131){6}
\pscircle[fillstyle=solid,fillcolor=MidRed,linecolor=black](253,136){6}
}
\NewOrderType 
\drawing{
\psline[linecolor=lightgray,linewidth=1pt](29,21)(107,235)
\psline[linecolor=lightgray,linewidth=1pt](107,235)(226,89)
\psline[linecolor=lightgray,linewidth=1pt](29,21)(226,89)
\pspolygon[fillstyle=solid,fillcolor=LightBlue,linecolor=black](107,235)(148,119)(109,65)
\pspolygon[fillstyle=solid,fillcolor=LightBlue,linecolor=black](107,235)(148,119)(226,89)
\pspolygon[fillstyle=solid,fillcolor=LightBlue,linecolor=black](148,119)(109,65)(226,89)
\pspolygon[fillstyle=solid,fillcolor=LightBlue,linecolor=black](29,21)(107,235)(109,65)
\pscircle[fillstyle=solid,fillcolor=MidRed,linecolor=black](29,21){6}
\pscircle[fillstyle=solid,fillcolor=MidRed,linecolor=black](80,104){6}
\pscircle[fillstyle=solid,fillcolor=MidRed,linecolor=black](107,235){6}
\pscircle[fillstyle=solid,fillcolor=MidRed,linecolor=black](109,65){6}
\pscircle[fillstyle=solid,fillcolor=MidRed,linecolor=black](148,119){6}
\pscircle[fillstyle=solid,fillcolor=MidRed,linecolor=black](226,89){6}
}
\drawing{
\psline[linecolor=lightgray,linewidth=1pt](29,21)(107,235)
\psline[linecolor=lightgray,linewidth=1pt](107,235)(226,89)
\psline[linecolor=lightgray,linewidth=1pt](29,21)(226,89)
\pspolygon[fillstyle=solid,fillcolor=LightBlue,linecolor=black](80,104)(148,119)(226,89)
\pspolygon[fillstyle=solid,fillcolor=LightBlue,linecolor=black](80,104)(109,65)(226,89)
\pscircle[fillstyle=solid,fillcolor=MidRed,linecolor=black](29,21){6}
\pscircle[fillstyle=solid,fillcolor=MidRed,linecolor=black](80,104){6}
\pscircle[fillstyle=solid,fillcolor=MidRed,linecolor=black](107,235){6}
\pscircle[fillstyle=solid,fillcolor=MidRed,linecolor=black](109,65){6}
\pscircle[fillstyle=solid,fillcolor=MidRed,linecolor=black](148,119){6}
\pscircle[fillstyle=solid,fillcolor=MidRed,linecolor=black](226,89){6}
}
\drawing{
\psline[linecolor=lightgray,linewidth=1pt](29,21)(107,235)
\psline[linecolor=lightgray,linewidth=1pt](107,235)(226,89)
\psline[linecolor=lightgray,linewidth=1pt](29,21)(226,89)
\pspolygon[fillstyle=solid,fillcolor=LightBlue,linecolor=black](80,104)(148,119)(109,65)
\pscircle[fillstyle=solid,fillcolor=MidRed,linecolor=black](29,21){6}
\pscircle[fillstyle=solid,fillcolor=MidRed,linecolor=black](80,104){6}
\pscircle[fillstyle=solid,fillcolor=MidRed,linecolor=black](107,235){6}
\pscircle[fillstyle=solid,fillcolor=MidRed,linecolor=black](109,65){6}
\pscircle[fillstyle=solid,fillcolor=MidRed,linecolor=black](148,119){6}
\pscircle[fillstyle=solid,fillcolor=MidRed,linecolor=black](226,89){6}
}
\drawing{
\psline[linecolor=lightgray,linewidth=1pt](29,21)(107,235)
\psline[linecolor=lightgray,linewidth=1pt](107,235)(226,89)
\psline[linecolor=lightgray,linewidth=1pt](29,21)(226,89)
\pspolygon[fillstyle=solid,fillcolor=LightBlue,linecolor=black](29,21)(107,235)(148,119)
\pspolygon[fillstyle=solid,fillcolor=LightBlue,linecolor=black](29,21)(148,119)(226,89)
\pscircle[fillstyle=solid,fillcolor=MidRed,linecolor=black](29,21){6}
\pscircle[fillstyle=solid,fillcolor=MidRed,linecolor=black](80,104){6}
\pscircle[fillstyle=solid,fillcolor=MidRed,linecolor=black](107,235){6}
\pscircle[fillstyle=solid,fillcolor=MidRed,linecolor=black](109,65){6}
\pscircle[fillstyle=solid,fillcolor=MidRed,linecolor=black](148,119){6}
\pscircle[fillstyle=solid,fillcolor=MidRed,linecolor=black](226,89){6}
}
\drawing{
\psline[linecolor=lightgray,linewidth=1pt](29,21)(107,235)
\psline[linecolor=lightgray,linewidth=1pt](107,235)(226,89)
\psline[linecolor=lightgray,linewidth=1pt](29,21)(226,89)
\pspolygon[fillstyle=solid,fillcolor=LightBlue,linecolor=black](107,235)(80,104)(226,89)
\pspolygon[fillstyle=solid,fillcolor=LightBlue,linecolor=black](29,21)(80,104)(226,89)
\pscircle[fillstyle=solid,fillcolor=MidRed,linecolor=black](29,21){6}
\pscircle[fillstyle=solid,fillcolor=MidRed,linecolor=black](80,104){6}
\pscircle[fillstyle=solid,fillcolor=MidRed,linecolor=black](107,235){6}
\pscircle[fillstyle=solid,fillcolor=MidRed,linecolor=black](109,65){6}
\pscircle[fillstyle=solid,fillcolor=MidRed,linecolor=black](148,119){6}
\pscircle[fillstyle=solid,fillcolor=MidRed,linecolor=black](226,89){6}
}
\drawing{
\psline[linecolor=lightgray,linewidth=1pt](29,21)(107,235)
\psline[linecolor=lightgray,linewidth=1pt](107,235)(226,89)
\psline[linecolor=lightgray,linewidth=1pt](29,21)(226,89)
\pspolygon[fillstyle=solid,fillcolor=LightBlue,linecolor=black](29,21)(107,235)(226,89)
\pscircle[fillstyle=solid,fillcolor=MidRed,linecolor=black](29,21){6}
\pscircle[fillstyle=solid,fillcolor=MidRed,linecolor=black](80,104){6}
\pscircle[fillstyle=solid,fillcolor=MidRed,linecolor=black](107,235){6}
\pscircle[fillstyle=solid,fillcolor=MidRed,linecolor=black](109,65){6}
\pscircle[fillstyle=solid,fillcolor=MidRed,linecolor=black](148,119){6}
\pscircle[fillstyle=solid,fillcolor=MidRed,linecolor=black](226,89){6}
}
\drawing{
\psline[linecolor=lightgray,linewidth=1pt](29,21)(107,235)
\psline[linecolor=lightgray,linewidth=1pt](107,235)(226,89)
\psline[linecolor=lightgray,linewidth=1pt](29,21)(226,89)
\pspolygon[fillstyle=solid,fillcolor=LightBlue,linecolor=black](107,235)(80,104)(148,119)
\pspolygon[fillstyle=solid,fillcolor=LightBlue,linecolor=black](29,21)(148,119)(109,65)
\pspolygon[fillstyle=solid,fillcolor=LightBlue,linecolor=black](29,21)(80,104)(148,119)
\pscircle[fillstyle=solid,fillcolor=MidRed,linecolor=black](29,21){6}
\pscircle[fillstyle=solid,fillcolor=MidRed,linecolor=black](80,104){6}
\pscircle[fillstyle=solid,fillcolor=MidRed,linecolor=black](107,235){6}
\pscircle[fillstyle=solid,fillcolor=MidRed,linecolor=black](109,65){6}
\pscircle[fillstyle=solid,fillcolor=MidRed,linecolor=black](148,119){6}
\pscircle[fillstyle=solid,fillcolor=MidRed,linecolor=black](226,89){6}
}
\drawing{
\psline[linecolor=lightgray,linewidth=1pt](29,21)(107,235)
\psline[linecolor=lightgray,linewidth=1pt](107,235)(226,89)
\psline[linecolor=lightgray,linewidth=1pt](29,21)(226,89)
\pspolygon[fillstyle=solid,fillcolor=LightBlue,linecolor=black](107,235)(109,65)(226,89)
\pspolygon[fillstyle=solid,fillcolor=LightBlue,linecolor=black](107,235)(80,104)(109,65)
\pspolygon[fillstyle=solid,fillcolor=LightBlue,linecolor=black](29,21)(107,235)(80,104)
\pspolygon[fillstyle=solid,fillcolor=LightBlue,linecolor=black](29,21)(80,104)(109,65)
\pspolygon[fillstyle=solid,fillcolor=LightBlue,linecolor=black](29,21)(109,65)(226,89)
\pscircle[fillstyle=solid,fillcolor=MidRed,linecolor=black](29,21){6}
\pscircle[fillstyle=solid,fillcolor=MidRed,linecolor=black](80,104){6}
\pscircle[fillstyle=solid,fillcolor=MidRed,linecolor=black](107,235){6}
\pscircle[fillstyle=solid,fillcolor=MidRed,linecolor=black](109,65){6}
\pscircle[fillstyle=solid,fillcolor=MidRed,linecolor=black](148,119){6}
\pscircle[fillstyle=solid,fillcolor=MidRed,linecolor=black](226,89){6}
}
\NewOrderType

\newpage
\section{13-Colouring the Triangles Determined by a Particular Set of 7 Points}
\label{sec:SevenPoints}
%\definecolor{MidRed}{rgb}{1,0.4,0.4}
%\definecolor{LightBlue}{rgb}{0.8,0.8,1}
\psset{unit=0.16mm}

\noindent
\drawing{
\psline[linecolor=lightgray,linewidth=1pt](7,159)(248,255)
\psline[linecolor=lightgray,linewidth=1pt](7,159)(124,1)
\psline[linecolor=lightgray,linewidth=1pt](124,1)(248,255)
\pspolygon[fillstyle=solid,fillcolor=LightBlue,linecolor=black](228,228)(68,160)(124,1)
\pspolygon[fillstyle=solid,fillcolor=LightBlue,linecolor=black](248,255)(228,228)(68,160)
\pspolygon[fillstyle=solid,fillcolor=LightBlue,linecolor=black](248,255)(68,160)(25,150)
\pspolygon[fillstyle=solid,fillcolor=LightBlue,linecolor=black](68,160)(25,150)(124,1)
\pspolygon[fillstyle=solid,fillcolor=LightBlue,linecolor=black](248,255)(228,228)(124,1)
\pspolygon[fillstyle=solid,fillcolor=LightBlue,linecolor=black](7,159)(25,150)(124,1)
\pspolygon[fillstyle=solid,fillcolor=LightBlue,linecolor=black](7,159)(248,255)(25,150)
\pscircle[fillstyle=solid,fillcolor=MidRed,linecolor=black](7,159){4}
\pscircle[fillstyle=solid,fillcolor=MidRed,linecolor=black](25,150){4}
\pscircle[fillstyle=solid,fillcolor=MidRed,linecolor=black](68,160){4}
\pscircle[fillstyle=solid,fillcolor=MidRed,linecolor=black](74,86){4}
\pscircle[fillstyle=solid,fillcolor=MidRed,linecolor=black](124,1){4}
\pscircle[fillstyle=solid,fillcolor=MidRed,linecolor=black](228,228){4}
\pscircle[fillstyle=solid,fillcolor=MidRed,linecolor=black](248,255){4}
}
\drawing{
\psline[linecolor=lightgray,linewidth=1pt](7,159)(248,255)
\psline[linecolor=lightgray,linewidth=1pt](7,159)(124,1)
\psline[linecolor=lightgray,linewidth=1pt](124,1)(248,255)
\pspolygon[fillstyle=solid,fillcolor=LightBlue,linecolor=black](228,228)(25,150)(124,1)
\pspolygon[fillstyle=solid,fillcolor=LightBlue,linecolor=black](7,159)(228,228)(25,150)
\pspolygon[fillstyle=solid,fillcolor=LightBlue,linecolor=black](7,159)(248,255)(228,228)
\pscircle[fillstyle=solid,fillcolor=MidRed,linecolor=black](7,159){4}
\pscircle[fillstyle=solid,fillcolor=MidRed,linecolor=black](25,150){4}
\pscircle[fillstyle=solid,fillcolor=MidRed,linecolor=black](68,160){4}
\pscircle[fillstyle=solid,fillcolor=MidRed,linecolor=black](74,86){4}
\pscircle[fillstyle=solid,fillcolor=MidRed,linecolor=black](124,1){4}
\pscircle[fillstyle=solid,fillcolor=MidRed,linecolor=black](228,228){4}
\pscircle[fillstyle=solid,fillcolor=MidRed,linecolor=black](248,255){4}
}
\drawing{
\psline[linecolor=lightgray,linewidth=1pt](7,159)(248,255)
\psline[linecolor=lightgray,linewidth=1pt](7,159)(124,1)
\psline[linecolor=lightgray,linewidth=1pt](124,1)(248,255)
\pspolygon[fillstyle=solid,fillcolor=LightBlue,linecolor=black](228,228)(25,150)(74,86)
\pspolygon[fillstyle=solid,fillcolor=LightBlue,linecolor=black](248,255)(228,228)(25,150)
\pspolygon[fillstyle=solid,fillcolor=LightBlue,linecolor=black](7,159)(25,150)(74,86)
\pspolygon[fillstyle=solid,fillcolor=LightBlue,linecolor=black](7,159)(74,86)(124,1)
\pspolygon[fillstyle=solid,fillcolor=LightBlue,linecolor=black](228,228)(74,86)(124,1)
\pscircle[fillstyle=solid,fillcolor=MidRed,linecolor=black](7,159){4}
\pscircle[fillstyle=solid,fillcolor=MidRed,linecolor=black](25,150){4}
\pscircle[fillstyle=solid,fillcolor=MidRed,linecolor=black](68,160){4}
\pscircle[fillstyle=solid,fillcolor=MidRed,linecolor=black](74,86){4}
\pscircle[fillstyle=solid,fillcolor=MidRed,linecolor=black](124,1){4}
\pscircle[fillstyle=solid,fillcolor=MidRed,linecolor=black](228,228){4}
\pscircle[fillstyle=solid,fillcolor=MidRed,linecolor=black](248,255){4}
}
\drawing{
\psline[linecolor=lightgray,linewidth=1pt](7,159)(248,255)
\psline[linecolor=lightgray,linewidth=1pt](7,159)(124,1)
\psline[linecolor=lightgray,linewidth=1pt](124,1)(248,255)
\pspolygon[fillstyle=solid,fillcolor=LightBlue,linecolor=black](228,228)(68,160)(25,150)
\pspolygon[fillstyle=solid,fillcolor=LightBlue,linecolor=black](68,160)(74,86)(124,1)
\pspolygon[fillstyle=solid,fillcolor=LightBlue,linecolor=black](68,160)(25,150)(74,86)
\pspolygon[fillstyle=solid,fillcolor=LightBlue,linecolor=black](25,150)(74,86)(124,1)
\pscircle[fillstyle=solid,fillcolor=MidRed,linecolor=black](7,159){4}
\pscircle[fillstyle=solid,fillcolor=MidRed,linecolor=black](25,150){4}
\pscircle[fillstyle=solid,fillcolor=MidRed,linecolor=black](68,160){4}
\pscircle[fillstyle=solid,fillcolor=MidRed,linecolor=black](74,86){4}
\pscircle[fillstyle=solid,fillcolor=MidRed,linecolor=black](124,1){4}
\pscircle[fillstyle=solid,fillcolor=MidRed,linecolor=black](228,228){4}
\pscircle[fillstyle=solid,fillcolor=MidRed,linecolor=black](248,255){4}
}

\drawing{
\psline[linecolor=lightgray,linewidth=1pt](7,159)(248,255)
\psline[linecolor=lightgray,linewidth=1pt](7,159)(124,1)
\psline[linecolor=lightgray,linewidth=1pt](124,1)(248,255)
\pspolygon[fillstyle=solid,fillcolor=LightBlue,linecolor=black](7,159)(228,228)(74,86)
\pscircle[fillstyle=solid,fillcolor=MidRed,linecolor=black](7,159){4}
\pscircle[fillstyle=solid,fillcolor=MidRed,linecolor=black](25,150){4}
\pscircle[fillstyle=solid,fillcolor=MidRed,linecolor=black](68,160){4}
\pscircle[fillstyle=solid,fillcolor=MidRed,linecolor=black](74,86){4}
\pscircle[fillstyle=solid,fillcolor=MidRed,linecolor=black](124,1){4}
\pscircle[fillstyle=solid,fillcolor=MidRed,linecolor=black](228,228){4}
\pscircle[fillstyle=solid,fillcolor=MidRed,linecolor=black](248,255){4}
}
\drawing{
\psline[linecolor=lightgray,linewidth=1pt](7,159)(248,255)
\psline[linecolor=lightgray,linewidth=1pt](7,159)(124,1)
\psline[linecolor=lightgray,linewidth=1pt](124,1)(248,255)
\pspolygon[fillstyle=solid,fillcolor=LightBlue,linecolor=black](7,159)(248,255)(124,1)
\pscircle[fillstyle=solid,fillcolor=MidRed,linecolor=black](7,159){4}
\pscircle[fillstyle=solid,fillcolor=MidRed,linecolor=black](25,150){4}
\pscircle[fillstyle=solid,fillcolor=MidRed,linecolor=black](68,160){4}
\pscircle[fillstyle=solid,fillcolor=MidRed,linecolor=black](74,86){4}
\pscircle[fillstyle=solid,fillcolor=MidRed,linecolor=black](124,1){4}
\pscircle[fillstyle=solid,fillcolor=MidRed,linecolor=black](228,228){4}
\pscircle[fillstyle=solid,fillcolor=MidRed,linecolor=black](248,255){4}
}
\drawing{
\psline[linecolor=lightgray,linewidth=1pt](7,159)(248,255)
\psline[linecolor=lightgray,linewidth=1pt](7,159)(124,1)
\psline[linecolor=lightgray,linewidth=1pt](124,1)(248,255)
\pspolygon[fillstyle=solid,fillcolor=LightBlue,linecolor=black](7,159)(228,228)(124,1)
\pscircle[fillstyle=solid,fillcolor=MidRed,linecolor=black](7,159){4}
\pscircle[fillstyle=solid,fillcolor=MidRed,linecolor=black](25,150){4}
\pscircle[fillstyle=solid,fillcolor=MidRed,linecolor=black](68,160){4}
\pscircle[fillstyle=solid,fillcolor=MidRed,linecolor=black](74,86){4}
\pscircle[fillstyle=solid,fillcolor=MidRed,linecolor=black](124,1){4}
\pscircle[fillstyle=solid,fillcolor=MidRed,linecolor=black](228,228){4}
\pscircle[fillstyle=solid,fillcolor=MidRed,linecolor=black](248,255){4}
}
\drawing{
\psline[linecolor=lightgray,linewidth=1pt](7,159)(248,255)
\psline[linecolor=lightgray,linewidth=1pt](7,159)(124,1)
\psline[linecolor=lightgray,linewidth=1pt](124,1)(248,255)
\pspolygon[fillstyle=solid,fillcolor=LightBlue,linecolor=black](248,255)(68,160)(74,86)
\pspolygon[fillstyle=solid,fillcolor=LightBlue,linecolor=black](248,255)(74,86)(124,1)
\pspolygon[fillstyle=solid,fillcolor=LightBlue,linecolor=black](7,159)(68,160)(25,150)
\pscircle[fillstyle=solid,fillcolor=MidRed,linecolor=black](7,159){4}
\pscircle[fillstyle=solid,fillcolor=MidRed,linecolor=black](25,150){4}
\pscircle[fillstyle=solid,fillcolor=MidRed,linecolor=black](68,160){4}
\pscircle[fillstyle=solid,fillcolor=MidRed,linecolor=black](74,86){4}
\pscircle[fillstyle=solid,fillcolor=MidRed,linecolor=black](124,1){4}
\pscircle[fillstyle=solid,fillcolor=MidRed,linecolor=black](228,228){4}
\pscircle[fillstyle=solid,fillcolor=MidRed,linecolor=black](248,255){4}
}

\drawing{
\psline[linecolor=lightgray,linewidth=1pt](7,159)(248,255)
\psline[linecolor=lightgray,linewidth=1pt](7,159)(124,1)
\psline[linecolor=lightgray,linewidth=1pt](124,1)(248,255)
\pspolygon[fillstyle=solid,fillcolor=LightBlue,linecolor=black](248,255)(68,160)(124,1)
\pspolygon[fillstyle=solid,fillcolor=LightBlue,linecolor=black](7,159)(248,255)(68,160)
\pspolygon[fillstyle=solid,fillcolor=LightBlue,linecolor=black](7,159)(68,160)(124,1)
\pscircle[fillstyle=solid,fillcolor=MidRed,linecolor=black](7,159){4}
\pscircle[fillstyle=solid,fillcolor=MidRed,linecolor=black](25,150){4}
\pscircle[fillstyle=solid,fillcolor=MidRed,linecolor=black](68,160){4}
\pscircle[fillstyle=solid,fillcolor=MidRed,linecolor=black](74,86){4}
\pscircle[fillstyle=solid,fillcolor=MidRed,linecolor=black](124,1){4}
\pscircle[fillstyle=solid,fillcolor=MidRed,linecolor=black](228,228){4}
\pscircle[fillstyle=solid,fillcolor=MidRed,linecolor=black](248,255){4}
}
\drawing{
\psline[linecolor=lightgray,linewidth=1pt](7,159)(248,255)
\psline[linecolor=lightgray,linewidth=1pt](7,159)(124,1)
\psline[linecolor=lightgray,linewidth=1pt](124,1)(248,255)
\pspolygon[fillstyle=solid,fillcolor=LightBlue,linecolor=black](248,255)(25,150)(124,1)
\pscircle[fillstyle=solid,fillcolor=MidRed,linecolor=black](7,159){4}
\pscircle[fillstyle=solid,fillcolor=MidRed,linecolor=black](25,150){4}
\pscircle[fillstyle=solid,fillcolor=MidRed,linecolor=black](68,160){4}
\pscircle[fillstyle=solid,fillcolor=MidRed,linecolor=black](74,86){4}
\pscircle[fillstyle=solid,fillcolor=MidRed,linecolor=black](124,1){4}
\pscircle[fillstyle=solid,fillcolor=MidRed,linecolor=black](228,228){4}
\pscircle[fillstyle=solid,fillcolor=MidRed,linecolor=black](248,255){4}
}
\drawing{
\psline[linecolor=lightgray,linewidth=1pt](7,159)(248,255)
\psline[linecolor=lightgray,linewidth=1pt](7,159)(124,1)
\psline[linecolor=lightgray,linewidth=1pt](124,1)(248,255)
\pspolygon[fillstyle=solid,fillcolor=LightBlue,linecolor=black](248,255)(228,228)(74,86)
\pspolygon[fillstyle=solid,fillcolor=LightBlue,linecolor=black](248,255)(25,150)(74,86)
\pscircle[fillstyle=solid,fillcolor=MidRed,linecolor=black](7,159){4}
\pscircle[fillstyle=solid,fillcolor=MidRed,linecolor=black](25,150){4}
\pscircle[fillstyle=solid,fillcolor=MidRed,linecolor=black](68,160){4}
\pscircle[fillstyle=solid,fillcolor=MidRed,linecolor=black](74,86){4}
\pscircle[fillstyle=solid,fillcolor=MidRed,linecolor=black](124,1){4}
\pscircle[fillstyle=solid,fillcolor=MidRed,linecolor=black](228,228){4}
\pscircle[fillstyle=solid,fillcolor=MidRed,linecolor=black](248,255){4}
}
\drawing{
\psline[linecolor=lightgray,linewidth=1pt](7,159)(248,255)
\psline[linecolor=lightgray,linewidth=1pt](7,159)(124,1)
\psline[linecolor=lightgray,linewidth=1pt](124,1)(248,255)
\pspolygon[fillstyle=solid,fillcolor=LightBlue,linecolor=black](228,228)(68,160)(74,86)
\pspolygon[fillstyle=solid,fillcolor=LightBlue,linecolor=black](7,159)(228,228)(68,160)
\pspolygon[fillstyle=solid,fillcolor=LightBlue,linecolor=black](7,159)(68,160)(74,86)
\pscircle[fillstyle=solid,fillcolor=MidRed,linecolor=black](7,159){4}
\pscircle[fillstyle=solid,fillcolor=MidRed,linecolor=black](25,150){4}
\pscircle[fillstyle=solid,fillcolor=MidRed,linecolor=black](68,160){4}
\pscircle[fillstyle=solid,fillcolor=MidRed,linecolor=black](74,86){4}
\pscircle[fillstyle=solid,fillcolor=MidRed,linecolor=black](124,1){4}
\pscircle[fillstyle=solid,fillcolor=MidRed,linecolor=black](228,228){4}
\pscircle[fillstyle=solid,fillcolor=MidRed,linecolor=black](248,255){4}
}
\drawing{
\psline[linecolor=lightgray,linewidth=1pt](7,159)(248,255)
\psline[linecolor=lightgray,linewidth=1pt](7,159)(124,1)
\psline[linecolor=lightgray,linewidth=1pt](124,1)(248,255)
\pspolygon[fillstyle=solid,fillcolor=LightBlue,linecolor=black](7,159)(248,255)(74,86)
\pscircle[fillstyle=solid,fillcolor=MidRed,linecolor=black](7,159){4}
\pscircle[fillstyle=solid,fillcolor=MidRed,linecolor=black](25,150){4}
\pscircle[fillstyle=solid,fillcolor=MidRed,linecolor=black](68,160){4}
\pscircle[fillstyle=solid,fillcolor=MidRed,linecolor=black](74,86){4}
\pscircle[fillstyle=solid,fillcolor=MidRed,linecolor=black](124,1){4}
\pscircle[fillstyle=solid,fillcolor=MidRed,linecolor=black](228,228){4}
\pscircle[fillstyle=solid,fillcolor=MidRed,linecolor=black](248,255){4}
}

\end{document}